\documentclass[11pt,reqno]{amsart}
\usepackage{fullpage,amssymb}
\usepackage{mathrsfs}
\usepackage{color}
\usepackage{algorithmicx}
\usepackage{algpseudocode}
\usepackage{tikz-cd}
\usepackage{enumitem}
\usepackage{xcolor}
\usepackage{hyperref}
\hypersetup{
    linktoc=all,     
    linkcolor=black,  
}
\usepackage{mathtools}

\DeclareMathOperator{\tr}{tr}
\DeclareMathOperator{\Gr}{Gr}
\DeclareMathOperator{\Stab}{Stab}
\newcommand{\idx}{\iota}
\newcommand{\gitsymbol}{\mathbin{
  \mathchoice{/\mkern-6mu/}
    {/\mkern-6mu/}
    {/\mkern-5mu/}
    {/\mkern-5mu/}}}
\makeatletter
\newcases{lrdcases}
  {\quad}
  {\hfil$\m@th\displaystyle{##}$\hfil}
  {\hfil$\m@th\displaystyle{##}$\hfil}
  {\lbrace}
  {\rbrace}
\makeatother
\newcommand{\gitPVG}{\mathbb PV\!\!\gitsymbol\!G}
\newtheorem{theorem}{Theorem}[section]
\newtheorem{corollary}[theorem]{Corollary}
\newtheorem{lemma}[theorem]{Lemma}

\theoremstyle{definition}
\newtheorem{definition}[theorem]{Definition}
\newtheorem{remark}[theorem]{Remark}

\newcommand{\diag}{{\rm diag}}
\newcommand{\R}{{\mathbb R}}

\newcommand{\PD} {{\rm PD}}
\newcommand{\mlt}{{\rm mlt}}

\newcommand{\Ten}{{\rm Ten}}
\newcommand{\Pp}{{\mathbb P}}

\newcommand{\C}{{\mathbb C}}

\newcommand{\F}{{\mathbb F}}
\newcommand{\Q}{{\mathbb Q}}
\newcommand{\Z}{{\mathbb Z}}

\newcommand{\Mat}{\operatorname{Mat}}
\newcommand{\GL}{\operatorname{GL}}
\newcommand{\SL}{{\rm SL}}

\newcommand{\semi}{{\rm ss}}
\newcommand{\ps}{{\rm ps}}
\newcommand{\st}{{\rm st}}
\allowdisplaybreaks[4]

\newcommand{\End}{{\rm End}}

\newcommand{\ad}{\rm ad}

 \newcommand{\gmax}{g_{\max}}

\title{Maximum likelihood estimation for tensor normal models via castling transforms}
\author{Harm Derksen, Visu Makam, and Michael Walter}
\thanks{HD was partially supported by NSF grants IIS-1837985 and DMS-2001460. VM was partially supported by the University of Melbourne and by NSF grants DMS-1638352 and CCF-1900460. MW acknowledges NWO Veni grant no.~680-47-459.}

\begin{document}
\begin{abstract}
In this paper, we study sample size thresholds for maximum likelihood estimation for tensor normal models.
Given the model parameters and the number of samples, we determine whether, almost surely,
(1) the likelihood function is bounded from above,
(2) maximum likelihood estimates (MLEs) exist, and
(3) MLEs exist uniquely.
We obtain a complete answer for both real and complex models.
One consequence of our results is that almost sure boundedness of the log-likelihood function guarantees almost sure existence of an MLE.
Our techniques are based on invariant theory and castling transforms.
\end{abstract}
\maketitle

\tableofcontents

\section{Introduction}
A family of probability distributions is called a statistical model.
Maximum likelihood estimation is a method of estimating the true probability distribution as the one that maximizes the likelihood of the observed data.
The probability distribution (or often the point in an associated parameter space) that maximizes the likelihood is called a \emph{maximum likelihood estimate (MLE)}.
One important problem is to understand the minimal number of samples required such that, almost surely,
(1) the likelihood function is bounded from above,
(2) MLEs exist, and
(3) there is a unique MLE.
Surprising connections between sample size thresholds for a class of models called Gaussian group models and stability notions in invariant theory were recently discovered in~\cite{AKRS}.
In this paper, we study sample size thresholds for \emph{tensor normal models}, which fall under the purview of Gaussian group models and are hence amenable to techniques from invariant theory.
The setting of invariant theory that relates to tensor normal models are the so-called tensor actions, i.e., the natural action of the group $\SL_{d_1} \times \SL_{d_2} \times \dots \times \SL_{d_k}$ on $\F^{d_1} \otimes \F^{d_2} \otimes \dots \otimes \F^{d_k}$, where $\F$ is the underlying field (either~$\R$ or $\C$) and $\SL_{d_i}$ denotes the group of $d_i \times d_i$ matrices with determinant one.

Tensor normal models are statistical models consisting of multivariate Gaussian distributions whose concentration matrix is a Kronecker (or tensor) product of several matrices.
These are particularly useful in studying data that naturally occurs as multi-dimensional arrays.
Examples include wood density in given growth rings and directions at several heights in a tree trunk~\cite{KZ}, monitoring of a vector of physiological variables in different organs over multiple days~\cite{RL}, and $3$-dimensional spatial glucose content data~\cite{Man-etal}. Moreover, tensors are ubiquitous in big data applications.

A special case of tensor normal models is the \emph{matrix normal model}, where the concentration matrix is a Kronecker product of exactly two matrices.
Sample size thresholds for matrix normal models have been investigated in~\cite{Dut99,Ros,Srivastava,Drton-etal,ST,AKRS,DM-mle}.
In particular, a complete answer for matrix normal models was obtained in~\cite{DM-mle} with techniques from quiver representations.
We do not use quiver representations in this paper, but instead we use  castling transforms and results on stabilizers in general position.
It is worth mentioning that the invariant theory for tensor actions with two tensor factors (which corresponds to the matrix normal models) is well understood and we have efficient algorithms, see~\cite{GGOW,DM,IQS,IQS2,DM-arbchar,DM-oc,AZGLOW}, whereas the invariant theory gets significantly more difficult for three and more tensor factors, see~\cite{BGOWW,BFGOWW,DM-exp} for more details.

To find the MLE, one can use the so called flip-flop algorithm~\cite{Dut99,LZ1,LZ2,Werner} for matrix normal models and its natural generalizations to tensor normal models, which is closely related to a recent alternating minimization algorithm in the invariant theory of tensor actions~\cite{AKRS,BGOWW,FORW}.
In general, MLEs for Gaussian group models can be found using the geodesic optimization algorithms in~\cite{BFGOWW}.

A separate motivation for studying the questions in this paper comes from quantum information.
Here tensors describe the states of a quantum mechanical system, and our invariant theoretic results characterize the existence of states with certain prescribed marginals, see~\cite{Klyachko,EntPoly,Walter,BRVR,BRVRquantum} for details.

\subsection{Tensor normal models}\label{subsec:models intro}
Let $\F = \R$ or $\C$.
Let $\PD_n$ denote the cone of positive definite $n \times n$ matrices with entries in $\F$. For an $n$-dimensional centered Gaussian distribution with concentration matrix $\Psi \in \PD_n$, the density function is defined as
\[
f_\Psi(y) = \det \left( \frac{\Psi}{2 \pi} \right)^{1/2} e^{-y^\dag \Psi y},
\]
where $y^\dag$ denotes the adjoint (conjugate transpose) of $y$.

Given a subset $\mathcal{M} \subseteq \PD_n$, we define the corresponding \emph{Gaussian model} as the statistical model consisting of the distributions with concentration matrix~$\Psi \in \mathcal M$.
Then the likelihood function $L_Y\colon \mathcal M \rightarrow \R$ is, for $m$ samples specified by an $m$-tuple $Y = (Y_1,\dots,Y_m) \in (\F^n)^m$, given by
\[
L_Y(\Psi) = \prod_{i=1}^m f_\Psi(Y_i) = \det \left(\tfrac{\Psi}{2\pi}\right)^{m/2} e^{-\frac{1}{2} \sum_{i=1}^m Y_i^\dag \Psi Y_i}.
\]
The log-likelihood function is then (up to an additive constant)
\begin{equation}\label{eq:l_Y intro}
l_Y(\Psi) = \frac{m}{2} \log \det (\Psi)  - \frac{1}{2} {\rm Tr} \left(\Psi \sum_{i=1}^m Y_i Y_i^\dag \right).
\end{equation}
A \emph{maximum likelihood estimate (MLE)} given~$Y$ is a concentration matrix~$\hat{\Psi} \in \mathcal{M}$ that maximizes the likelihood of observing the data~$Y$, i.e., $l_Y(\hat{\Psi}) \geq l_Y(\Psi)$ for all $\Psi \in \mathcal{M}$.
For an MLE to exist, it is therefore necessary (but not necessarily sufficient) that~$l_Y$ is bounded from above.
Even when they exist, MLEs need not be unique.

For $d_1,\dots,d_k \in \Z_{>0}$, the Gaussian model $\mathcal{M}(d_1,\dots,d_k) = \{\Psi_1 \otimes \Psi_2 \otimes \dots \otimes \Psi_k\ |\ \Psi_i \in \PD_{d_i} \} \subseteq \PD_n$ (where $n = d_1 d_2 \cdots d_k$) is called a \emph{tensor normal model}.
When we want to differentiate between the real and the complex model, we will write $\mathcal{M}_\R(d_1,\dots,d_k)$ and $\mathcal{M}_\C(d_1,\dots,d_k)$ respectively.
For the tensor normal model $\mathcal{M}(d_1,\dots,d_k)$, a sample can not only be thought of as a vector of size~$n$, but also as a $k$-tensor with local dimensions $d_1,d_2,\dots,d_k$.
The latter viewpoint will be particularly useful.
Accordingly, we define $\F^{d_1,\dots,d_k} \coloneqq \F^{d_1} \otimes \F^{d_2} \otimes \dots \otimes \F^{d_k}$.
Then a sample for the tensor normal model $\mathcal{M}(d_1,\dots,d_k)$ is simply a point in the tensor space $\F^{d_1,\dots,d_k}$. We also write $\F^{d_1,\dots,d_k;m}$ for $(\F^{d_1,\dots,d_k})^{\oplus m}$.

\subsection{Main results on sample size thresholds}
Generalizing the quantity $R(d_1,\dots,d_k)$ defined in~\cite{BRVR}, we consider
\begin{align*}
  R(d_1,\dots,d_k;m) \coloneqq m \prod_{i=1}^k d_i + \sum_{n=1}^k (-1)^n \sum_{1\leq i_1 < \ldots < i_n \leq k} \gcd(d_{i_1},\dots,d_{i_n})^2,
\end{align*}
as well as the following two quantities:
\begin{align*}
  \gmax(d_1,\dots,d_k) \coloneqq \max_{i<j} \gcd(d_i,d_j), \qquad
  \Delta(d_1,\dots,d_k; m) \coloneqq m \prod_{i=1}^k d_i - 1 - \sum_{i=1}^k (d_i^2 - 1).
\end{align*}
By convention, $\gmax(d) = 1$ for any $d\in\Z_{>0}$.
Then all three quantities are invariant under leaving out dimensions equal to one.
The following theorem shows that these quantities precisely predict the almost sure behavior of the MLE.
By almost surely we mean that the stated property holds for all~$Y$ away from a subset of $\Ten(d_1,\dots,d_k)^{\oplus m} \cong (\F^n)^m$ of Lebesgue measure zero.

\begin{theorem}\label{thm:main}
Let $\F = \R$ or $\C$.
Consider $m$ samples $Y = (Y_1,\dots,Y_m)$ of the tensor normal model $\mathcal{M}(d_1,\dots,d_k)$.
Let $R = R(d_1,\dots,d_k;m)$, $\Delta = \Delta(d_1,\dots,d_k;m)$, and $\gmax = \gmax(d_1,\dots,d_k)$.~Then:
\begin{enumerate}
\item If $R > 0$, then almost surely an MLE exists. Furthermore:
\begin{itemize}
\item If $m \geq 2$, 
the MLE is almost surely unique if and only if $R > \gmax^2$ or $\gmax\!=\!1$.
\item If $m = 1$, 
the MLE is almost surely unique if and only if $\Delta \geq -1$.
\end{itemize}
\item If $R = 0$, then almost surely an MLE exists. It is almost surely unique if and only if~$\gmax\!=\!1$.
\item If $R < 0$, then the likelihood function is always unbounded from above.
\end{enumerate}
\end{theorem}

\begin{remark}
It was conjectured in~\cite{Drton-etal} and proved in~\cite{DM-mle} that for matrix normal models (tensor normal models with $k=2$), almost sure boundedness of the log-likelihood function implies almost sure existence of an MLE.
Theorem~\ref{thm:main} implies that the same holds for all tensor normal models.
\end{remark}

From Theorem~\ref{thm:main}, we can extract the following result.
Let us denote by $\mlt_b$ (resp.\ $\mlt_e$, $\mlt_u$) the smallest integer~$m_0$ such that, for all $m \geq m_0$, the log-likelihood function for $Y \in \F^{d_1,\dots,d_k;m}$ is almost surely bounded from above (resp.\ MLEs exist, the MLE exists uniquely).

\begin{corollary}\label{cor:threshold}
Let $\F = \R$ or $\C$.
Consider the tensor normal model $\mathcal{M}(d_1,\dots,d_k)$.
Without loss of generality, assume $2 \leq d_1 \leq d_2 \leq \dots \leq d_k$, and assume $k \geq 3$.
Let $r = \frac{d_k}{d_1d_2\cdots d_{k-1}}$.
Then
\[
\lceil r \rceil \leq \mlt_b = \mlt_e \leq \mlt_u \leq \lceil r \rceil + 1.
\]
\end{corollary}

We note that for the case $k = 2$, a complete answer is known~\cite{DM-mle}.
The case $k = 1$ is trivial.
Corollary~\ref{cor:threshold} gives nearly tight bounds on sample size thresholds.
However, we note that for any particular choice of $d_1,\dots,d_k$, we can always use Theorem~\ref{thm:main} to get exact sample size thresholds.

\subsection{Main results in invariant theory}
Recently, Am\'endola, Kohn, Reichenbach and Seigal~\cite{AKRS} established a connection between a class of Gaussian models called Gaussian group models and the invariant theory of a corresponding group action (see Theorem~\ref{theo:AKRS}).
We revisit this connection in Section~\ref{sec:gg-inv}.
As mentioned previously, the group action that corresponds to tensor normal models is the tensor action.
Given natural numbers $d_1,\dots,d_k,m\in \Z_{>0}$, we denote by $\rho_{d_1,\dots,d_k;m}$ the natural representation of $G = \SL_{d_1}(\F) \times \cdots \times \SL_{d_k}(\F)$ on $V = \F^{d_1,\dots,d_k;m}$.
Theorem~\ref{thm:main} is a consequence of the following invariant-theoretic result (see Section~\ref{sec:gg-inv} for the definitions of stability).

\begin{theorem}\label{thm:main-inv}
Let $\F = \R$ or $\C$.
Consider the tensor representation $\rho = \rho_{d_1,\cdots,d_k;m}$.
Let $R = R(d_1,\dots,d_k;m)$, $\Delta = \Delta(d_1,\dots,d_k;m)$, and $\gmax = \gmax(d_1,\dots,d_k)$.
Then:
\begin{enumerate}
\item If $R > 0$, then $\rho$ is generically polystable. Furthermore:
\begin{itemize}
\item If $m\geq2$, then $R \geq \gmax^2$, and $\rho$ is generically stable if and only if $R > \gmax^2$ or~$\gmax=1$.
\item If $m = 1$, then $\Delta \geq -2$, and $\rho$ is generically stable if and only if $\Delta \geq -1$.
\end{itemize}
\item If $R = 0$, then $\rho$ is generically polystable. It is generically stable if and only if~$\gmax = 1$.
\item If $R < 0$, then $\rho$ is unstable.
\end{enumerate}
\end{theorem}

While the preceding theorem gives a nice and uniform characterization, it is essentially a reformulation of the following result which is \emph{recursive} in nature, but more enlightening.

\begin{theorem}\label{thm:recursive}
Let $\F = \R$ or $\C$.
Consider the tensor representation $\rho = \rho_{d_1,\cdots,d_k;m}$.
Without loss of generality, assume $d_1 \leq d_2 \leq \dots \leq d_k$.
Then:
\begin{enumerate}
\item If $d_k > d_1 \cdots d_{k-1} m$, then $\rho$ is not generically semistable.
\item If $d_k = d_1 \cdots d_{k-1} m$, then $\rho$ is generically polystable.
It is generically stable if and only if $d_1 = \cdots = d_{k-1} = 1$.
\item\label{it:castling} If $\frac{d_1 \cdots d_{k-1} m}2 < d_k < d_1 \cdots d_{k-1} m$, then $\rho$ is generically semistable (polystable, stable) if and only if the same is true if we replace~$d_k$ by $d'_k = d_1 \cdots d_{k-1} m - d_k$.
Note that $1 \leq d'_k < d_k$.
\item If $d_k \leq \frac{d_1 \cdots d_{k-1} m}2$, then $\rho$ is generically polystable.
Further, it is not generically stable if and only if $(d_1,\dots,d_k;m) =(1,\dots,1,2,d,d;1)$ or $(1,\dots,1,1,d,d;2)$ for some $d\geq2$.
\end{enumerate}
Moreover, if $\rho$ is not generically semistable then it is unstable.
\end{theorem}

\noindent
Part~(\ref{it:castling}) of Theorem~\ref{thm:recursive} above is a reflection of the fact that the property of being generically semistable (polystable, stable) is unchanged under an operation known as a \emph{castling transform}.
Castling transforms played a crucial role in Sato and Kimura's classification of prehomogeneous vector spaces~\cite{SK} (see also~\cite{Venturelli}).
Its origins can be traced back to at least Elashvili's paper~\cite{Elashvili}.

As a corollary of Theorems~\ref{thm:main-inv} and~\ref{thm:recursive}, we can derive a formula for the dimension of the GIT quotient (see Section~\ref{sec:git quotient} for definition) of $V = \F^{d_1,\dots,d_k;m}$ for the action of $G = \SL_{d_1} \times \dots \times \SL_{d_k}$.
This generalizes the result of~\cite{BRVR}, where the dimension was computed in the case that $m=1$.

\begin{theorem}\label{thm:git}
Let $\F=\C$.
Consider the natural action of $G = \SL_{d_1} \times \cdots \times \SL_{d_k}$ on $V = \F^{d_1,\dots,d_k;m}$.
Let $\delta$ denote the dimension of the GIT quotient $\gitPVG$.
\begin{enumerate}
\item If $R < 0$, then 
the GIT quotient is empty.
\item If $R = 0$, then $\delta = 0$. In fact, the GIT quotient is a single point.
\item If $R > 0$, then
\begin{align*}
\delta = \begin{cases}
\max(\gmax-3,0) & \text{ if $m = 1$ and $\Delta = -2$}, \\
\gmax & \text{ if $m = 2$ and $R = \gmax^2 > 1$}, \\
\Delta & \text{ otherwise}.
\end{cases}
\end{align*}
\end{enumerate}
\end{theorem}

\subsection*{Organization of the paper}
In Section~\ref{sec:gg-inv}, we revisit the general connection between Gaussian group models and invariant theory, and discuss the relevant notions of stability.
In Section~\ref{sec:castling}, we introduce castling transforms and discuss how they preserve stability.
In Section~\ref{sec:stability}, this is used as the key ingredient to derive our recursive characterization (Theorem~\ref{thm:recursive}).
In Section~\ref{sec:uniform}, we deduce our uniform characterization (Theorem~\ref{thm:main-inv}) from the former.
In Section~\ref{sec:mle}, we prove our main results on sample size thresholds for tensor normal models (Theorem~\ref{thm:main} and Corollary~\ref{cor:threshold}).
Finally, in Section~\ref{sec:git quotient} we compute the dimension of the GIT quotient (Theorem~\ref{thm:git}).

\subsection*{Acknowledgements} We would like to thank Carlos Am\'endola, Suguman Bansal, Christian Ikenmeyer, Kathl\'en Kohn, Siddharth Krishna, Mark Van Raamsdonk, Philipp Reichenbach, and Anna Seigal for interesting discussions.


\section{Gaussian group models and invariant theory} \label{sec:gg-inv}
In this section we first discuss the general setup of invariant theory.
Then we define Gaussian group models and their connection to notions of generic stability in invariant theory.
Finally, we discuss some general criteria from the literature useful for characterizing generic stability.

Let $\F = \R$ or $\C$.
Let $G$ be a group.
A \emph{representation} of $G$ is an action of $G$ on a (finite-dimensional) vector space $V$ (over $\F$) by linear transformations.
This is captured succinctly as a group homomorphism $\rho\colon G \rightarrow \GL(V)$.
In particular, an element $g \in G$ acts on $V$ by the linear transformation $\rho(g)$.
We write $g \cdot v$ or $gv$ to mean $\rho(g)v$.
The $G$-orbit of $v \in V$ is the set of all vectors that you can get from $v$ by applying elements of the group, i.e.,
\[
O_v \coloneqq \{gv\ |\ g \in G\} \subseteq V.
\]

Throughout this paper, we will only consider the setting where $G$ is a linear algebraic group (over~$\F$) and where the action is rational, i.e., $\rho\colon G \rightarrow \GL(V)$ is a morphism of algebraic groups.

We denote by $\F[V]$ the ring of polynomial functions on $V$ (also known as the coordinate ring of~$V$).
A polynomial function $f \in \F[V]$ is called {\em invariant} if $f(gv) = f(v)$ for all $g \in G$ and $v \in V$.
In other words, a polynomial is called invariant if it is constant on orbits.
The invariant ring is
\[
\F[V]^G \coloneqq \{f \in \F[V]\ |\ f(gv) = f(v) \ \forall\ g \in G, v \in V\}.
\]
The invariant ring has a natural grading by degree, i.e., $\F[V]^G = \oplus_{d=0}^\infty \F[V]^G_d$ where $\F[V]^G_d$ consists of all invariant polynomials that are homogeneous of degree $d$.
For $v \in V$, we define the stabilizer subgroup $G_v \coloneqq \{g \in G \ |\ gv = v\}$ and we denote by $\overline{O}_v$, the closure of the orbit $O_v$.

\begin{remark}
To define the closure, we need to specify a topology on~$V$.
In this paper, we only use the fields $\F = \R$ or $\C$.
Hence, we will use the standard Euclidean topology on $V$ for orbit closures, unless otherwise specified.
At times we will also need to use the Zariski topology, but we will be careful in specifying it each time.
For $\F = \C$, the orbit closure w.r.t.\ the Euclidean topology agrees with the orbit closure w.r.t.\ the Zariski topology (in the setting of rational actions of reductive groups).
\end{remark}

We make a few definitions, the significance of which will become clear in the following subsections.

\begin{definition}\label{def:stability}
Let $\F = \R$ or $\C$, and let $G$ be an algebraic group (over $\F$) with a rational action on a vector space $V$ (over $\F$), given by $\rho\colon G \rightarrow \GL(V)$.
Let $K$ denote the kernel of the homomorphism~$\rho$. Give $V$ the standard Euclidean topology. Then, for $v \in V$, we say $v$ is
\begin{itemize}
\item {\em unstable} if $0 \in \overline{O}_v$;
\item  {\em semistable} if $0 \notin \overline{O}_v$;
\item  {\em polystable} if $v \neq 0$ and $O_v$ is closed;
\item  {\em stable} if $v$ is polystable and the quotient $G_v/K$ is finite.
\end{itemize}
\end{definition}

\subsection{Gaussian group models}
For a subgroup $G \subseteq \GL_n$, we define an associated \emph{Gaussian group model} by the following family of concentration matrices:
\[
\mathcal{M}_G \coloneqq \{g^\dag g\ |\ g \in \GL_n\}.
\]
where $g^\dagger = \bar g^T$ denotes the adjoint.
So for a concentration matrix $\Psi = g^\dag g \in \mathcal{M}_G$ and an $m$-tuple of samples $Y = (Y_1,\dots,Y_m) \in (\F^n)^m$, the log-likelihood function~\eqref{eq:l_Y intro} simplifies to
\[
l_Y(\Psi) = \frac{m}{2} \log(\det(g^\dag g))  - \frac{1}{2} \lVert g \cdot Y\rVert^2,
\]
where $\lVert\cdot\rVert$ denotes the $\ell_2$-norm on $(\F^n)^m \cong \F^{nm}$ and we note that $G$ acts on $(\F^n)^m$ by the diagonal action $g \cdot Y = (g Y_1,\dots,g Y_m)$.

The following result was proved in \cite[Theorems~6.10 and 6.24]{AKRS}.
It connects maximum likelihood estimation in Gaussian group models to the stability notions introduced in Definition~\ref{def:stability}.

\begin{theorem}[\cite{AKRS}]\label{theo:AKRS}
Let $\F = \R$ or $\C$.
Let $G \subseteq \GL_n$ be a Zariski-closed subgroup that is closed under adjoints and non-zero scalar multiples.
Let $G_{\SL} = \{g \in G \ | \det(g) = 1\} \subseteq G$ and let $Y \in (\F^n)^m$ be an $m$-tuple of samples.
Then, for the diagonal action of $G_{\SL}$, we have
\begin{itemize}
\item  $Y$ is semistable $\Longleftrightarrow$ $l_Y$ is bounded from above;
\item $Y$ is polystable $\Longleftrightarrow$ an MLE exists (i.e., $l_Y$ has a maximum);
\item $Y$ is stable $\implies$ there exists a unique MLE (i.e., $l_Y$ has a unique maximum). \\
If $\F = \C$, the converse also holds, i.e., there exists a unique MLE $\implies Y$ is stable.
\end{itemize}
Moreover, if $\Psi$ is an MLE given $Y$, then the set of all MLEs given $Y$ is $\{g^\dag \Psi g \ |\ g \in (G_{\SL})_Y\}$.
\end{theorem}

\begin{remark}
In the setting of the above theorem, for $h \in G_{\SL}$, we also have
\[
\bigl\{\text{MLEs given } h \cdot Y\bigr\} = (h^{-1})^\dag \bigl\{\text{MLEs given } Y\bigr\} h^{-1}.
\]
Thus, for any $h \in G_{\SL}$, the MLE given $Y$ is unique if and only if the MLE given $h \cdot Y$ is unique.
\end{remark}

Now let $\F = \R$ and suppose $Y$ is already a point with minimal norm in its orbit.
Then for an appropriate $\lambda \in \R_{>0}$, we have that $\lambda I$ is an MLE and the set of all MLEs is $\{\lambda g^\dag g \ |\ g \in (G_{\SL})_Y\}$.
In particular, we have a unique MLE if and only if $(G_{\SL})_Y \subseteq O_n$, the orthogonal group.
Further, since $(G_{\SL})_Y$~is closed, it must be compact.
The stabilizer of any other point in its $G_\SL$-orbit is obtained by conjugation and remains compact.
In particular, if $Y$ is any tuple of samples such that the MLE exists uniquely, then $(G_{\SL})_Y$ is compact.
This will be important to us, so we record the statement for later use:

\begin{corollary} \label{cor:uniq->compact}
Suppose we are in the setting of Theorem~\ref{theo:AKRS}, with $\F = \R$.
If the MLE given $Y$ exists uniquely, then $(G_{\SL})_Y$ is compact.
\end{corollary}

\noindent
When $\F=\C$, the same hypothesis and argument shows that $(G_{\SL})_Y$ is finite.
However, we will only need Corollary~\ref{cor:uniq->compact} in the case that $\F=\R$.

\subsection{Notions of generic stability}
Let $G$ be an algebraic group (over $\F)$ and let $V$ be a rational representation (over $\F$). Then we define:
\begin{align*}
V^{\semi} & = \{v \in V \ |\ v \text{ is $G$-semistable}\}, \\
V^{\ps} & = \{v \in V \ |\ v \text{ is $G$-polystable}\}, \\
V^{\st} & = \{v \in V \ |\ v \text{ is $G$-stable}\}.
\end{align*}
We call $V^{\semi}$ (resp.\ $V^{\ps}, V^{\st}$) the \emph{semistable} (resp.\ \emph{polystable}, \emph{stable}) \emph{locus}. If the group is not clear from context then we write $V^{G\text{-}\semi}, V^{G\text{-}\ps}, V^{G\text{-}\st}$.

\begin{definition} \label{defn.gen.stable}
Let $\F = \R$ or $\C$, and let $G$ be an algebraic group (over $\F$) with a rational action on a vector space $V$ (over $\F$).
Then, we say $V$ is \emph{generically $G$-semistable} (resp.\ \emph{polystable}, \emph{stable}) if $V^{\semi}$ (resp.\ $V^{\ps}, V^{\st}$) contain a non-empty Zariski-open subset of $V$.
Further, we say that $V$ is \emph{unstable} if $V^{\semi} = \emptyset$.
\end{definition}

These notions are particularly well-behaved in the case that $\F=\C$, as we will see in the following.
We refer to \cite[Corollary~2.15, Lemma~2.16]{DM-mle} for a succinct proof of the following standard result:

\begin{lemma}\label{lem.loc.const}
Suppose $\F = \C$. Let $V$ be a rational representation of a complex reductive group $G$.
Then, the subsets $V^{\semi}$ and $V^{\st}$ are Zariski-open and the subset $V^{\ps}$ is Zariski-constructible, i.e., it is a union of Zariski-locally closed subsets.
Moreover, $V$ is generically semistable if and only if it is not unstable.
\end{lemma}

Zariski-open subsets of a vector space, whenever non-empty, are complements of lower dimensional subvarieties, which have Lebesgue measure zero.
Zariski-constructible subsets of a vector space, on the other hand, have Lebesgue measure zero unless they contain a Zariski-open subset, in which case their complement has Lebesgue measure zero.
Hence, we can conclude the following:

\begin{corollary}\label{cor:mle-gen.stable}
Suppose we are in the setting of Theorem~\ref{theo:AKRS}, with $\F = \C$.
Fix a number of samples~$m$ and let $V=(\C^n)^m$.
Then, for the diagonal action of $G_\SL$ we have
\begin{itemize}
\item $V$ is generically semistable $\Longleftrightarrow$ $l_Y$ is almost surely bounded from above
\item $V$ is generically polystable $\Longleftrightarrow$ an MLE exists almost surely;
\item $V$ is generically stable $\Longleftrightarrow$ there exists a unique MLE almost surely;
\item $V$ is unstable $\Longleftrightarrow$ $l_Y$ is always unbounded from above.
\end{itemize}
Moreover, the first and last condition are complementary.
Here we say a property holds almost surely if it holds for all $Y$ in $V$ up to a set of Lebesgue measure zero.
\end{corollary}

Let us also mention one lemma that will be useful for us later

\begin{lemma}\label{lem.increase}
Suppose $G$ is a complex algebraic group and let $V$ be a rational representation over~$\C$.
If $V^{\oplus m}$ is generically $G$-stable (resp.\ $G$-semistable), then $V^{\oplus n}$ is generically $G$-stable (resp.\ $G$-semistable) for all $n \geq m$ with respect to the diagonal actions of $G$.
\end{lemma}
\begin{proof}
Suppose $V^{\oplus m}$ is generically $G$-stable.
We have an inclusion $(V^{\oplus m})^{\st} \subseteq (V^{\oplus n})^{\st}$ with respect to the diagonal actions of $G$.
So $(V^{\oplus n})^{\st}$ is non-empty and further it is Zariski open by Lemma~\ref{lem.loc.const}.
Thus, $V^{\oplus n}$ is generically $G$-stable.
The argument for semistability is similar.
\end{proof}

\subsection{Stabilizers in general position}\label{subsec:sgp}
Let $\F = \C$ for this section.
Let $V$ be a rational representation of a reductive group $G$.
We say that $H$ is a \emph{stabilizer in general position (s.g.p.)} if there is a non-empty Zariski-open subset $U \subseteq V$ such that for all $v \in U$, the stabilizer $G_v$ is isomorphic to $H$.
The s.g.p.\ is unique up to conjugation.
Its existence is far from obvious and follows from Luna's slice theorem, see e.g., \cite[Theorem~7.2]{Popov-Vinberg}.
Indeed, when $\F = \R$, stabilizers in general position often do not exist.

Matsushima's criterion tells us that if an orbit of a point is closed, then the stabilizer is reductive.
Hence, if $V$ is generically polystable, then the s.g.p.\ must be reductive. The converse was proved by Popov:

\begin{theorem}[\cite{Popov}]\label{theo:Popov}
Let $\rho\colon G \rightarrow \GL(V)$ be a rational representation of a reductive group. Then,~$V$ is generically polystable if and only if the stabilizer in general position is reductive.
\end{theorem}

\begin{corollary}\label{cor:crit gen stab}
Let $\rho\colon G \rightarrow \GL(V)$ be a rational representation of a reductive group and let $K$ denote the kernel of $\rho$. Let $H$ be the stabilizer in general position. The following are equivalent.
\begin{enumerate}
\item $V$ is generically stable;
\item $\dim(H) = \dim(K)$;
\item $\dim(G_v) = \dim(K)$ for some $v \in V$;
\end{enumerate}
\end{corollary}

\begin{proof}
Clearly $(1) \implies (2) \implies (3)$.
For $(2) \implies (1)$, Observe that $\dim(H) = \dim(K)$ implies that $G_v/K$ is finite for generic $v \in V$. The kernel of a morphism of (affine) algebraic groups between reductive groups is reductive, so $K$ is reductive. Since $G_v/K$ is finite  (for generic $v \in V$), this means that $G_v$ and $K$ have the same identity component and hence $G_v$ is also reductive. In particular, it means that $H$ is reductive. Hence $V$ is generically polystable by Theorem~\ref{theo:Popov}, and of further generically stable because $G_v/K$ is finite for generic $v \in V$.

For $(3) \implies (2)$, we observe that the set of points $U = \{v \in V\ |\ \dim(G_v) \leq \dim(K)\}$ is Zariski open.
Note that $U = \{v \in V\ |\ \dim(G_v) = \dim(K)\}$ since $K \subseteq G_v$ for all $v \in V$.
Since $U$ is non-empty Zariski open, it follows that $\dim(H) = \dim(K)$ as well.
\end{proof}

\subsection{A criterion for generic (poly)stability}\label{subsec:index}
Let still be $\F = \C$ for this section.
Starting from the late 1960s, there has been an interest in classifying actions that are generically polystable or stable, see for example, \cite{ave,Elashvili,SK,AMPopov}.
From this line of research, we will recall a few results that will be important for us.

If $S$ is a simple algebraic group, then the Killing form defined by $(X,Y)\mapsto \tr(\ad(X)\ad(Y))$ is a nondegenerate symmetric $S$-invariant bilinear form on the Lie algebra ${\mathfrak s}$ of $S$.
Up to a scalar,~${\mathfrak s}$~has only one $S$-invariant symmetric bilinear form. If $\rho\colon S\to \GL(V)$
and $d\rho\colon {\mathfrak s}\to \End(V)$ is the corresponding representation of the Lie algebra, then $(X,Y)\mapsto\tr(d\rho(X)d\rho(Y))$ is a nonzero symmetric $S$-invariant bilinear form on~${\mathfrak s}$.
So there is a constant $\iota_S(V)$, called the \emph{index} of the representation, such that
\begin{align*}
  \tr(d\rho(X)d\rho(Y))=\iota_S(V)\tr(\ad(X)\ad(Y))  
\end{align*}
for all $X,Y\in {\mathfrak s}$.
The index is additive.
Furthermore, we have $\idx_{\SL_n}(\C^n) = \frac1{2n}$ for the defining representation of $\SL_n$.

Andreev, Vinberg, and Elashvili proved the following criterion for generic stability in~\cite[Theorem]{ave}.

\begin{theorem}[\cite{ave}]\label{thm:ave}
Let $\rho\colon G \rightarrow \GL(V)$ be a rational representation of a connected semisimple\footnote{Semisimple groups are reductive.}
group.
Let $H$ be the stabilizer in general position.
If $\idx_S(V) > 1$ for all simple normal subgroups $S \subseteq G$, then $\dim(H) = 0$.
In particular, $V$ is generically $G$-stable.
\end{theorem}

Elashvili proved a very similar criterion for generic polystability in \cite[Theorem 2]{Elashvili}.

\begin{theorem}[\cite{Elashvili}]\label{thm:elashvili}
Let $\rho\colon G \rightarrow \GL(V)$ be a rational representation of a connected semisimple
group.
Let $H$ be the stabilizer in general position.
If $\idx_S(V) \geq 1$ for all simple normal subgroups $S\subseteq G$, then the Lie algebra of $H$ is the Lie algebra of a torus.
In particular, $H$ is reductive, so $V$ is generically $G$-polystable.
\end{theorem}

Just to put these results in context, let us consider the tensor action, i.e., the action of $G = \smash{\prod_{i=1}^k \SL_{d_k}}$ on $\C^{d_1,\dots,d_k;m}$. In this case, $G$ is a connected semisimple group and its simple normal subgroups are just $\SL_{d_1},\SL_{d_2},\dots,\SL_{d_k}$ and the index for each $\SL_{d_i}$ is $\frac{m \prod_{j \neq i} d_j}{2 d_i}$.

Finally, Elashvili has classified all irreducible representations that satisfy the hypotheses of Theorem~\ref{thm:elashvili} but are not generically stable, see \cite[Theorem 9]{Elashvili} and Theorem~\ref{thm:elashvili classification} below.

\section{Castling transforms}\label{sec:castling}
In this section, we take $\F = \R$ or $\C$.
Let $\rho\colon G \rightarrow \GL(V)$ be an $n$-dimensional representation of an algebraic group $G$.
We will assume $\rho(G) \subseteq \SL(V)$.
For $0 < k < n$, we have a natural action of $G \times \SL_k$ on $V \otimes \F^k$, where $G$ acts on $V$ and $\SL_k$ acts on $\C^k$.
Similarly, we have an action of $G$ on $V^*$ and $\SL_{n-k}$ on $\F^{n-k}$, which together gives an action of $G \times \SL_{n-k}$ on $V^* \otimes \F^{n-k}$.
We refer to the action of $G \times \SL_{n-k}$ on $V^* \otimes \F^{n-k}$ as a {\em castling transform} of the action of $G \times \SL_k$ on $V \otimes \F^k$.

The main feature of castling transforms is that we get a bijection between the $G \times \SL_k$-orbits in a non-empty Zariski-open subset of $V \otimes \F^k$ and the $G \times \SL_{n-k}$-orbits in a non-empty Zariski-open subset of $V^* \otimes \F^{n-k}$.
Moreover, this bijection of orbits preserves stabilizers up to isomorphism.
Hence, when $\F = \C$, the stabilizer in general position is preserved under castling transforms.
Moreover, generic semistability/polystability/stability will also be preserved under castling transforms.
We will now explain all this in more detail, but first we need to recall Grassmannians.

\subsection{Grassmannians}
Let $\F = \R$ or $\C$.
Suppose $V$ is an $n$-dimensional vector space over $\F$.
Let~$\Gr(k,V)$ denote the Grassmannian of $k$-planes in $V$.
It is naturally embedded in $\Pp(\smash{\bigwedge^k(V)})$ as a closed subvariety cut out by the Pl\"ucker relations, where $\smash{\bigwedge^k(V)}$ denotes the $k^{\text{th}}$ exterior power of~$V$.

This embedding is constructed as follows.
Identify $V$ with $\F^n$ by choosing a basis $e_1,\dots,e_n$.
Then, a basis for $\smash{\bigwedge^k(V)}$ is $\{e_{i_1} \wedge e_{i_2} \wedge \dots \wedge e_{i_k} \ |\ 1 \leq i_1 < i_2 < \dots < i_k \leq n\}$.
For any subset~$I \subseteq [n]$ of size $k$, we write $e_I$ to denote $e_{i_1} \wedge e_{i_2} \wedge \dots \wedge e_{i_k}$ where $I = \{i_1,\dots,i_k\}$ with the $i_j$'s in increasing order.
We write $\Delta_I$ to denote the coordinate corresponding to $e_I$.
Now, for any subspace $L$ of $V$ of dimension $k$, take independent vectors $l_1,\dots,l_k$ in $L$ and consider the point $[l_1 \wedge l_2 \wedge \dots \wedge l_k] \in \Pp(\smash{\bigwedge^k(V)})$.
This point is independent of the choice of $l_i$ and only depends on the subspace $L$.
Thus, we obtain an injective map $\Gr(k,V) \rightarrow \Pp(\smash{\bigwedge^k(V)})$ whose image is a closed subvariety.
This map is called the Pl\"ucker embedding and endows the Grassmannian with the structure of a projective variety.
We refer to \cite{Fulton,Weyman,Procesi-book} for more details on Grassmannians.

The affine cone over the Grassmannian $\widehat\Gr(k,V)$ is a closed subvariety of $\smash{\bigwedge^k(V)}$.
Note that $\smash{\widehat\Gr(k,V)} = \{v_1 \wedge v_2 \wedge \dots \wedge v_k \ |\ v_i \in V\}$.
If the $v_i$'s are linearly dependent, then $v_1 \wedge v_2 \wedge \dots \wedge v_k = 0$, otherwise it is nonzero.
Let $\{e_1,\dots,e_k\}$ denote the standard basis for $\F^k$, and define
\begin{align}\label{eq:U}
  U = \left\{{\textstyle\sum_{i=1}^k} v_i \otimes e_i \in V \otimes \F^k \ | \ v_1,\dots,v_k \text{ are linearly independent}\right\}.
\end{align}
Then, we have a map
\begin{align*}
\pi = \pi_{k,V} \colon U \longrightarrow \widehat\Gr(k,V) \setminus \{0\}, \qquad
{\textstyle\sum_{i=1}^k v_i \otimes e_i} \longmapsto v_1 \wedge v_2 \wedge \dots \wedge v_k.
\end{align*}

We claim that $U$ is a Zariski-locally trivial principal $\SL_k$-bundle over $\widehat\Gr(k,V) \setminus \{0\}$.
It is straightforward to see that it is a principal $\SL_k$-bundle, because $v_1 \wedge v_2 \wedge \dots \wedge v_k = w_1 \wedge w_2 \wedge \dots \wedge w_k$ if and only if there is a matrix $A = (a_{ij}) \in \SL_k$ such that $\sum_i a_{ij} v_j = w_i$ for all $i$.
To see that is Zariski-locally trivial needs an explanation.
A similar result, namely that $U$ is a Zariski-locally trivial principal $\GL_k$-bundle over $\Gr(k,V)$ is well known, see e.g., \cite[pg.~511]{Procesi-book}.
We modify their argument appropriately.

First, we note that $\widehat\Gr(k,V) \setminus \{0\}$ is covered by affine open subsets $\{X_I : I \subseteq [n], |I| = k\}$, where $X_I \coloneqq \{p\ |\ \Delta_I(p) \neq 0\}$.
If we identify $V$ with $\F^n$ as mentioned above, $U$ can be viewed as the $k \times n$ matrices of full rank.
For a matrix $M \in \Mat_{k,n}$, and a subset $I \subseteq [n]$ of size $k$, let $M_I$ denote the $k \times k$ submatrix of $M$ obtained by considering the columns labeled by elements in $I$, and let $p_I(M) = \det(M_I)$.
Then $\pi^{-1}(X_I) = \{M \in \Mat_{k,n} \ |\ p_I(M) \neq 0\}$.
Without loss of generality, we can take $I = \{1,2,\dots,k\}$, so we have an isomorphism $\pi^{-1}(X_I) \rightarrow \Mat_{k,n-k} \times \F^* \times \SL_k$ given by $M = [A \ |\ B] \mapsto (DA^{-1}B ,\det(A), AD^{-1})$ where $D$ is the diagonal matrix with diagonal entries $(\det(A), 1,1,\dots,1)$. The map in the reverse direction is $(P,\lambda,Q) \mapsto [Q D \ |\ QP]$ where $D = {\rm diag}(\lambda,1,\dots,1)$. Next, observing that $X_I \cong \Mat_{k,n-k} \times \F^*$\footnote{It is well known in the projective setting that the locus where $\Delta_I(p) \neq 0$ is isomorphic to $\Mat_{k,n-k}$, and we are just pulling back to the affine cone.} gives us an isomorphism $\pi^{-1}(X_I) \longrightarrow X_I \times \SL_k$.

Everything we said above also works if you consider the Euclidean topology because Zariski-open subsets are open in the Euclidean topology and polynomial maps are continuous in the Euclidean topology as well.
Hence, $U$ is a locally trivial principal $\SL_k$-bundle over $\widehat\Gr(k,V) \setminus \{0\}$ in the Euclidean topology as well.

The projection of a locally-trivial bundle onto its base is an open map. One can check this condition on a trivializing cover of the base. In other words, it suffices to check that projection of a trivial bundle onto its base is open. For the Euclidean topology, it is well known that projection maps are open. For the Zariski topology, projection maps are also open. When the underlying field is algebraically closed, this follows from flatness, but remains true even when the underlying field is not algebraically closed, see Appendix~\ref{app:projection} for a proof.

To summarize, we get the following result:

\begin{lemma} \label{lem.quot.map}
Let $V$, $U$, and $\pi_{k,V}$ be defined as above.
Then $U$ is a Zariski-locally trivial principal $\SL_k$-bundle over $\smash{\widehat\Gr(k,V)} \setminus \{0\}$ via the map $\pi_{k,V}$.
In particular, $\pi_{k,V}$ is an open map (and also a quotient map) when considering either the Zariski or Euclidean topology.
\end{lemma}


\subsection{Castling transforms}
Let $\rho\colon G \rightarrow \GL(V)$ be a representation of an algebraic group $G$ and we will assume $\rho(G) \subseteq \SL(V)$.
Let $\dim(V) = n$.
We have an action of $G \times \SL_k$ on $V \otimes \F^k$ and an action of $G \times \SL_{n-k}$ on $V^* \otimes \F^{n-k}$. Let
\[
U = \left\{{\textstyle\sum_{i =1}^k} v_i \otimes e_i \in V \otimes \F^k\ | \ v_1,\dots,v_k \text{ are linearly independent} \right\} \subseteq V \otimes \F^k,
\]
as in~\eqref{eq:U} and let
\[
U' = \left\{{\textstyle\sum_{i=1}^{n-k}} w_i \otimes e_i \in V^* \otimes \F^{n-k}\ | \ w_1,\dots,w_{n-k} \text{ are linearly independent} \right\} \subseteq V^* \otimes \F^{n-k}.
\]
Since $U$ is a principal $\SL_k$-bundle over $\widehat\Gr(k,V) \setminus \{0\}$, we have a bijection between the $\SL_k$-orbits in~$U$ and the points of $\widehat\Gr(k,V) \setminus \{0\}$.
This bijection is $G$-equivariant since $\pi_{k,V}$ is $G$-equivariant 
and the actions of $G$ and of $\SL_k$ on $V \otimes \C^k$ commute.
So, we have $G$-equivariant bijections:
\begin{align}\label{eq:bijections}
  \text{$\SL_k$-orbits in $U$}
\ \longleftrightarrow\ %
\widehat{\Gr}(k,V) \setminus \{0\}
\ \longleftrightarrow\ %
\widehat{\Gr}(n-k,V^*) \setminus \{0\}
\ \longleftrightarrow\ %
\text{$\SL_{n-k}$-orbits in $U'$}
\end{align}
The first bijection was explained above and the last bijection follows by the same argument.
The middle bijection comes from the well understood $\SL(V)$-equivariant isomorphism $\smash{\bigwedge^k(V)} \cong \smash{\bigwedge^{n-k}(V^*)}$.
The following result is implicit in \cite{Elashvili}, but we furnish a proof for completeness.

\begin{lemma}\label{stab.equal}
Let $T \in U$.
Then, we have an isomorphism of algebraic groups
\[
\Stab_G(\pi_{k,V}(T)) \cong \Stab_{G \times \SL_k} (T).
\]
\end{lemma}
\begin{proof}
This holds since $\pi_{k,V}$ is a $G$-equivariant principal $\SL_k$-bundle.
Indeed, let $p\colon G \times \SL_k \rightarrow G$ denote the projection onto the first factor.
It is easy to see that $p(\Stab_{G \times \SL_k} (T)) \subseteq \Stab_G(\pi_{k,V}(T))$.
Now suppose $g \in \Stab_G (\pi_{k,V}(T))$.
Then, $\pi_{k,V}(T) = g \cdot \pi_{k,V}(T)$ implies that $\pi_{k,V}(T) = \pi_{k,V}(g \cdot T)$ by $G$-equivariance.
Since $\pi_{k,V}$ is a principal $\SL_k$-bundle, it follows that there exists a unique $A\in\SL_k$ such that $A \cdot (g \cdot T) = T$, i.e., $(g,A) \cdot T = T$.
Thus we have proved that every $g \in \Stab_G(\pi_{k,V}(T))$ has a unique preimage under $p$ in $\Stab_{G \times \SL_k} (T)$.
We conclude that $p$ restricted to $\Stab_{G \times \SL_k} (T)$ is a (group) isomorphism onto its image, which is $\Stab_G(\pi_{k,V}(T))$.

To establish that this is an isomorphism of algebraic groups (over $\F$), we need to establish that it is an isomorphism of varieties.
To do so, we give a map in the reverse direction as follows.
Write~$T = \sum_{i=1}^k v_i \otimes e_i$.
Let $g \in \Stab_G(\pi_{k,V}(T))$.
Since $g$ stabilizes the span of $v_1,\dots,v_k$, we get that $g \cdot v_i = \sum_j c_{i,j}(g) \, v_j$ where the $c_{i,j}(g)$ are regular functions on $\Stab_G(\pi_{k,V}(T))$.
Moreover, the matrix $C = (c_{i,j}(g))_{1\leq i,j \leq k}$ is invertible.
Then $(g,\smash{C^{-1}})$ is the unique preimage of $g$ in $\Stab_{G \times \SL_k} (T)$ under $p$.
Thus the map $g \mapsto (g,\smash{C^{-1}})$ is the inverse of $p$ restricted to $\Stab_{G \times \SL_k}(T)$, and it is clearly a morphism of algebraic varieties.
\end{proof}

As a consequence of the bijections~\eqref{eq:bijections} and Lemma~\ref{stab.equal}, we thus obtain the following corollaries.

\begin{corollary}\label{cor.cas.stab.pres}
We have a natural bijection between the $G\times \SL_k$-orbits in $U$ and the $G \times \SL_{n-k}$ orbits in $U'$ that preserves stabilizers (up to isomorphism).
\end{corollary}

\begin{corollary}\label{cor.cas.stab}
Let $\F = \C$.
Then the stabilizer in general position for the action of $G \times \SL_k$ on $V \otimes \C^k$ is isomorphic to the stabilizer in general position for the action of $G \times \SL_{n-k}$  on $V^* \otimes \C^{n-k}$.
\end{corollary}

In fact, the invariant ring is also preserved by castling transforms~\cite{SK} (see also~\cite[Prop.~2.1]{Kac}).

\begin{lemma}[\cite{SK}]\label{lem.cas.inv.ring}
Let $\F = \C$.
Then the invariant ring for the action of $G \times \SL_k$ on $V \otimes \C^k$ is (canonically) isomorphic to the invariant ring for the action of $G \times \SL_{n-k}$  on $V^* \otimes \C^{n-k}$.
\end{lemma}

The discussion above culminates in the following result that will be very important for us:

\begin{corollary}\label{cor:castling}
Let $\F = \R$ or $\C$.
Then $V \otimes \F^k$ is generically $G \times \SL_k$-semistable (polystable, stable) if and only if $V^* \otimes \F^{n-k}$ is generically $G \times \SL_{n-k}$-semistable (polystable, stable).
\end{corollary}

\begin{proof}
By \cite[Proposition~2.23]{DM-mle}, it suffices to prove the statement for $\F = \C$.
So, let us assume that~$\F = \C$.
Generic semistability is the same as having a non-trivial invariant ring.
Hence, it follows from Lemma~\ref{lem.cas.inv.ring} that castling transforms preserve generic semistability.
The fact that castling transforms preserve generic polystability follows from Corollary~\ref{cor.cas.stab} and Theorem~\ref{theo:Popov}.

That castling transforms preserve generic stability follows similarly from Corollaries~\ref{cor.cas.stab} and~\ref{cor:crit gen stab}, provided we can show that the kernels of the two actions have the same dimension.
To see this, let $K = \ker(\rho)$, where $\rho\colon G \rightarrow \GL(V)$ is the action of $G$ on $V$.
Now, let us consider the kernel of $\tilde{\rho}\colon G \times \SL_k \rightarrow \GL(V \otimes \C^k)$.
For $(g,A) \in G \times \SL_k$, we have $\tilde{\rho}(g,A) = \rho(g) \otimes A$.
So, if $(g,A)$ is in the kernel, then $\rho(g) = c \rm I$ and $A =  c^{-1} I$ for some $c \in \C^*$.
But $A \in \SL_k$, so $c$ must be an $k^{\text{th}}$ root of unity. For each such $c$, the subvariety $H_c = \{g \in G \ |\ \rho(g) = cI\}$ is either empty or a coset of~$K$.
Since the kernel is a finite union of $H_c \times \{c^{-1} I\}$, its dimension equals the dimension of $K$.
On the other hand, the kernel for the action of $G$ on $V^*$ is also $K$, so the same argument shows that the kernel for the action of $G \times \SL_{n-k}$ on $V^* \otimes \C^{n-k}$ also has the same dimension as~$K$.
\end{proof}

For complex Gaussian group models, we saw in Theorem~\ref{theo:AKRS} that invariant-theoretic stability notions characterize the boundedness of the log-likelihood function and the existence and uniqueness of MLEs precisely.
However, for real models, the relation between generic stability and almost sure existence of a unique MLE is less tight.
To bridge this gap, we will need the following results:

\begin{lemma}\label{lem:open castling}
Suppose $P \subseteq V \otimes \F^k$ is an open subset in the Euclidean (resp.\ Zariski) topology, then $(\pi_{n-k,V^*})^{-1} \pi_{k,V} (P \cap U)$ is a non-empty open subset of $U'$ in the Euclidean (resp.\ Zariski) topology.
\end{lemma}
\begin{proof}
Let us first argue this for Euclidean topology.
Observe that $P \cap U$ is an open subset of $V \otimes \C^k$.
Further, since $U^c$ is a proper subvariety and hence has empty interior, we know that $P \cap U$ must be non-empty. Now, the statement follows since $\pi_{k,V}$ is an open map by Lemma~\ref{lem.quot.map}.
The argument for Zariski topology is analogous.
\end{proof}

An immediate corollary of the above lemma is the following:

\begin{corollary}\label{lem:castle compact}
Let $P = \{T \in V \otimes \F^k\ |\ \Stab_{G \times \SL_k}(T) \text{ is not compact} \}$.
Similarly, let $P' = \{S \in V^* \otimes \F^{n-k}\ |\ \Stab_{G \times \SL_{n-k}}(T) \text{ is not compact}\}$.
Then $P$ contains a Euclidean (resp.\ Zariski) open subset of $V \otimes \F^k$ if and only if $P'$ contains a Euclidean (resp.\ Zariski) open subset of $V^* \otimes \F^k$.
\end{corollary}
\begin{proof}
It suffices to prove one direction.
Suppose $P$ contains a non-empty Euclidean (resp. Zariski) open subset $\widetilde{P}$.
Then, by Lemma~\ref{lem:open castling}, $(\pi_{n-k,V^*})^{-1} \pi_{k,V} (\widetilde{P} \cap U)$ is a Euclidean (resp. Zariski) open subset of $V^* \otimes \F^{n-k}$, and it is contained in $P'$ by Corollary~\ref{cor.cas.stab.pres}.
\end{proof}

We need to give a technical clarification in the above corollary with respect to notion of compactness.
There are two natural topologies one can give a Lie subgroup $H$ of a Lie group $G$.
The first is the inherent topology on $H$ by virtue of being a Lie group in itself, and the second is the subspace topology by virtue of being a subspace of $G$.
In the proof above, we are really using the inherent topology because the isomorphism of stabilizers furnished by Corollary~\ref{cor.cas.stab.pres} is an abstract isomorphism.
However, we will later need to use the lemma in the context of Corollary~\ref{cor:uniq->compact}, which refers to the subspace topology.
While for immersed Lie subgroups the inherent topology can differ from the subspace topology, the two topologies coincide for embedded Lie subgroups.
Since stabilizer subgroups are closed, they are embedded Lie subgroups and there is no ambiguity.

\subsection{Castling transforms for tensor actions}
We now discuss explicitly the relevance of castling transforms to tensor actions and hence to tensor normal models.
Here we are interested in the action of $\prod_{i=1}^k \SL_{d_i}$ on $\F^{d_1,\dots,d_k;m}$, which we succinctly denote by $\rho_{d_1,\dots,d_k;m}$. 
The ground field~$\F$ is assumed to be either $\R$ or $\C$.
If we need to specify it, we will add a subscript.

Let $G = \prod_{i=1}^{k-1} \SL_{d_i}$ and consider its natural action on $V = \F^{d_1,\dots,d_{k-1};m}$, which in our notation is $\rho_{d_1,\dots,d_{k-1};m}$.
Then, the action of $G \times \SL_{d_k}$ on $V \otimes \F^{d_k}$ is simply $\rho_{d_1,\dots,d_k;m}$.
It is well known that $V$ and $V^*$ are related by an automorphism on the group~$G$, which does not affect any of the notions of stability.%
\footnote{If we compose a representation~$\rho$ of~$\SL(d)$ with the automorphism $g \mapsto g^{-T}$, the result is isomorphic to the dual representation of $\rho$, and similarly for the product group $G$.}
Hence, we call $\rho_{d_1,\dots,N-d_k;m}$ the \emph{castling transform} of $\rho_{d_1,\dots,d_k;m}$, where~$N = \dim V = md_1\cdots d_{k-1}$ and we assume that $N > d_k$.
Thus Corollary~\ref{cor:castling} implies the following important result:

\begin{corollary}\label{cor.ten.cas.pres}
Let $d_1,\dots,d_k,m \in \Z_{>0}$ and suppose that $N = m \prod_{i=1}^{k-1} d_i > d_k$.
Then, $\rho_{d_1,\dots,d_k;m}$ is generically semistable (polystable, stable) if and only if $\rho_{d_1,\dots,N-d_k;m}$ is generically semistable (polystable, stable).
\end{corollary}

Given this result, we will make some definitions for later use.
For positive integers $d_1,\dots,d_k$ and~$m$, we call $(d_1,\dots,d_k;m)$ a \emph{datum} and $\rho_{d_1,\dots,d_k;m}$ the corresponding representation.
Observe that permuting the $d_i$ leaves the group and representation unchanged up to isomorphism, hence does not change the generic stability properties of the representation.

\begin{definition}
We say two data $(d_1,\dots,d_k;m)$ and $(d_1',\dots,d_k';m)$ are \emph{castling-equivalent} if $\rho_{d_1,\dots,d_k;m}$ and $\rho_{d'_1,\dots,d'_k;m}$ are related by a sequence of castling transforms (of the form described above) and permutations of the dimensions.
We say the datum $(d_1,\dots,d_k;m)$ is \emph{minimal} in its castling equivalence class if it minimizes $\prod_{i=1}^k d_i$.
\end{definition}


\begin{lemma}\label{lem:min-castle}
Consider the datum $(d_1,\dots,d_k;m)$.
Without loss of generality, we assume that $d_1 \leq d_2 \leq \dots \leq d_k$.
Let $N = m \cdot \smash{\prod_{i=1}^{k-1}} d_i$.
Then, if $\smash{\frac N2} < d_k < N$, the datum is not minimal in its castling equivalence class.
\end{lemma}
\begin{proof}
We only need to show that if $\frac N2 < d_k < N$, then the datum is not minimal.
To see this, observe that we have a castling transform that takes $(d_1,\dots,d_k;m)$ to $(d_1,\dots,d_{k-1}, N-d_k;m)$ and the latter is smaller since $N - d_k < d_k$.
\end{proof}

\begin{remark}\label{rmk:d1 redundant}
If $d_1 = 1$, then $\rho_{d_1,d_2,\dots,d_k;m}$ and $\rho_{d_2,\dots,d_k;m}$ are equal up to isomorphism of the group and representation, so we can often assume without loss of generality that $d_i \geq 2$.
\end{remark}

Even though it will not be relevant to us, we observe that each castling equivalence class contains a unique minimal datum (up to permutation).
This follows from the fact that if any two data are related by (minimal) sequence of castling transforms, then the sequence of dimensions of representations produced by these transforms is monotonous, the proof of which is exactly the same as the proof of \cite[Proposition~29]{Manivel}.

\section{Stability for tensor actions}\label{sec:stability}
In this section, we will prove Theorem~\ref{thm:recursive}, which gives a recursive characterization of the generic stability properties for the tensor actions $\rho_{d_1,\cdots,d_k;m}$.
Without loss of generality, we may assume that $d_1 \leq d_2 \leq \dots \leq d_k$.
By Corollary~\ref{cor.ten.cas.pres}, we know that the properties we are looking are invariant under the castling transform in part~(3) of the theorem, so the majority of our work will be spent on the terminal cases.
We now prove each part of the theorem separately.

For the first part, we need a simple lemma.
It follows from the first fundamental theorem of invariant theory for the special linear group, a result that dates back to Weyl~\cite{Weyl}, but also has an elementary proof (see also \cite[p.~7, Example]{KP}).

\begin{lemma}\label{lem:fft}
Consider the action of $G = \SL_d$ on $V = \Mat_{d,r}$ by left multiplication.
If $d > r$, then every point $v \in V$ is $G$-unstable.
In contrast, if $d \leq r$, then $V$ is generically $G$-stable.
\end{lemma}
\begin{proof}
Suppose $d>r$.
Then, any $v \in \Mat_{d,r}$ has rank at most $r$, so we can find $g\in\SL_d$ such that the range of $g v$ is a subspace of the span of the first $r<d$ standard basis vectors.
Then, $\varphi(t) := g^{-1} \diag(t^{d-r},\dots,t^{d-r}, t^{-r}, \dots,t^{-r}) g \in \SL_d$ for all~$t\neq0$, and $\varphi(t) v \to 0$ as $t\to0$.

Now suppose that $d \leq r$.
By Lemma~\ref{lem.increase}, it suffices to prove the claim in the case that $d=r$.
Suppose $v \in \Mat_{d,d}$ is invertible (a Zariski-open set).
Then its $\SL_d$-orbit is equal to $\det^{-1}(\det v)$, hence closed.
Since moreover its stabilizer is trivial, we conclude that $V$ is generically stable.

Note that Lemma~\ref{lem.increase} was stated only for $\F =\C$.
There are many ways to adapt the argument for $\F = \R$, e.g., one can use \cite[Proposition~2.23]{DM-mle}.
\end{proof}

\begin{proof}[Proof of Theorem~\ref{thm:recursive}, part (1)]
As a representation of $\SL_{d_k}$, the tensor space $\F^{d_1,\dots,d_k;m}$ is isomorphic to $\Mat_{d_k,md_1d_2\cdots d_{k-1}}$ and hence every point is unstable by Lemma~\ref{lem:fft}, since $d_k > d_1 \cdots d_{k-1} m$.
Hence every point is also unstable for the action of the larger group $G = \smash{\prod_{i=1}^k \SL_{d_i}}$.
\end{proof}

For the second part, we will need the following result.

\begin{lemma} \label{lem:stab.case2}
Let $\pi\colon H \rightarrow \SL_d \subseteq \GL_d$ be a $d$-dimensional representation of an algebraic group~$H$.
Consider the action of $G = H \times \SL_d$ on $\Mat_{d,d}$ given by $(h,g) \cdot A = \pi(h) A g^{-1}$.
For any full-rank matrix $A \in \Mat_{d,d}$, the stabilizer is given by $G_A = \{(h,A^{-1}\pi(h)A) \ |\ h \in H\}$.
In particular, the stabilizer in general position is isomorphic to~$H$.
\end{lemma}
\begin{proof}
Straightforward.
\end{proof}

One point to note is that the kernel of the tensor action $\rho = \rho_{d_1,\dots,d_k;m}$ is finite.
So stability is equivalent to having a closed orbit and finite stabilizer.
In particular, for $\F = \C$, Corollary~\ref{cor:crit gen stab} shows that generic stability of $\rho$ is the same as the stabilizer in general position being finite.

\begin{proof}[Proof of Theorem~\ref{thm:recursive}, part (2), for $\F = \C$]
Let us define $H = \SL_{d_1} \times \SL_{d_2} \times \dots \times \SL_{d_{k-1}}$ and $W = \C^{d_1,\dots,d_{k-1};m}$.
Then we can view $G \cong H \times \SL_{d_k}$ and $\C^{d_1,\dots,d_k;m} \cong W \otimes \C^{d_k} \cong \Mat_{d_k,d_k}$, since $d_k = d_1 \cdots d_{k-1} m$.
So, the stabilizer in general position is $H$ by Lemma~\ref{lem:stab.case2}, which is reductive.
Hence, $\rho = \rho_{d_1,\dots,d_k;m}$ is generically polystable by Theorem~\ref{theo:Popov}.
As discussed above, the kernel of $\rho$ is a finite group, so $\rho$ is generically stable if and only if the stabilizer in general position~$H$ is finite.
This happens precisely when $d_1 = d_2 = \dots = d_{k-1} = 1$.
\end{proof}

\begin{proof}[Proof of Theorem~\ref{thm:recursive}, part (2), for $\F = \R$]
This follows from \cite[Proposition~2.23]{DM-mle}.
\end{proof}

We already proved the third part of the theorem when we discussed the castling transforms of tensor actions.

\begin{proof} [Proof of Theorem~\ref{thm:recursive}, part (3)]
This follows from Corollary~\ref{cor.ten.cas.pres}.
\end{proof}

We now prove the fourth and last part of the theorem, which is perhaps the most complicated.
Here we wish to apply Theorems~\ref{thm:ave} and~\ref{thm:elashvili}.
Recall from Section~\ref{subsec:index} that the simple normal subgroups of $G = \SL_{d_1} \times \SL_{d_2} \times \dots \times \SL_{d_k}$ are just $\SL_{d_1},\SL_{d_2},\dots,\SL_{d_k}$.
To compute the index of $V = \C^{d_1,\dots,d_k;m}$ with respect to some~$\SL_{d_i}$, note that $V \cong (\C^{d_i})^{\oplus M}$ as an $\SL_{d_i}$-representation, where $M = \frac{md_1\cdots d_k}{d_i}$.
Now, the index of $\C^{d_i}$ with respect to $\SL_{d_i}$ is $\frac{1}{2d_i}$ and is additive.
It follows that the index of $\C^{d_1,\dots,d_k;m}$ with respect to $\SL_{d_i}$ is given by $\smash{\frac{M}{2d_i} = \frac{md_1d_2\cdots d_k}{2d_i^2}}$.
Since by assumption $d_1 \leq d_2 \leq \dots \leq d_k$, the smallest of these indices is the one for $\SL_{d_k}$, given by $\smash{\frac{md_1d_2\cdots d_{k-1}}{2d_k}}$.
When~$d_k \leq \frac12 m d_1 d_2 \cdots d_{k-1}$, as we assume in part~(4) of the theorem, all indices therefore are at least one, so Theorems~\ref{thm:ave} and~\ref{thm:elashvili} are applicable.

When $m=1$, then the representation of~$G$ on~$V$ is irreducible.
Elashvili has classified all irreducible representations of semisimple groups which are generically polystable, but not generically stable.
From the classification one can extract the following, see \cite[Theorem 9]{Elashvili} and also~\cite[p.~9]{BRVR}.

\begin{theorem}[\cite{Elashvili}]\label{thm:elashvili classification}
Consider the irreducible representation $V = \C^{d_1,\dots,d_k;m}$ of $G=\SL_{d_1}(\C) \times \cdots \times \SL_{d_k}(\C)$.
Assume that $2 \leq d_1 \leq \cdots \leq d_k \leq \frac{d_1 \cdots d_{k-1}}2$.
Then, $V$ satisfies the hypotheses of Theorem~\ref{thm:elashvili}, hence is generically $G$-polystable.
Moreover, $V$ is not generically $G$-stable if and only if $k=3$ and $(d_1,d_2,d_3) = (2,d,d)$ for some $d\geq2$.
\end{theorem}

Note that this result proves part~(4) of the theorem when $\F = \C$ and $m=1$.
To deal with the case that $m\geq2$, we will still use of this theorem, together with a knowledge of the s.g.p.'s.

For $(d_1,d_2,d_3)=(2,2,2)$, the stabilizer of $v=e_1^{\otimes 3}+e_2^{\otimes 3}$ is a s.g.p.
It includes and has the same Lie algebra as the two-dimensional torus
$\{ (s,t,u) \in G : s, t, u \text{ diagonal}, stu = 1 \}$.

For $(d_1,d_2,d_3)=(2,d,d)$, $d>2$, the stabilizer of $v = e_1 \otimes I + e_2 \otimes A$, where $I$ denotes the $d\times d$ identity matrix and $A$ is a generic $d\times d$ diagonal matrix, is a s.g.p.
It includes and has the same Lie algebra as the $(d-1)$-dimensional torus $\{ (1,t,t^{-1}) \in G : t \text{ diagonal} \}$.

\begin{proof} [Proof of Theorem~\ref{thm:recursive}, part (4) for $\F = \C$]
Since $d_k \leq \smash{\frac{d_1 \cdots d_{k-1} m}2}$, the index of~$V = \C^{d_1,\dots,d_k;m}$ with respect to any simple normal subgroup of~$G$ is greater than or equal to one (as discussed above).
When the inequality is strict, then $\rho$ is generically stable by Theorem~\ref{thm:ave}.
Now suppose that $d_k = \smash{\frac{d_1 \cdots d_{k-1} m}2}$.
Then $\rho$ is still generically polystable by Theorem~\ref{thm:elashvili}.
We now characterize when the representation is generically stable.
If $d_1 = \cdots = d_{k-1} = 1$ then $d_k \leq \frac m2 < m$, so $\rho$ is generically stable by Lemma~\ref{lem:fft}.
Now assume that $d_{k-1} \geq 2$.
Then, $d_k = \frac{m d_1 \cdots d_{k-1}}2 \geq m$.
This means that if we consider the action of the larger group $H = \SL_{d_1} \times \dots \times \SL_{d_k} \times \SL_m$ on $V = \C^{d_1} \otimes \dots \otimes \C^{d_k} \otimes \C^m$ then the dimension $d_k$ is still the largest among the dimensions $d_1,d_2,\dots,d_k,m$.
Accordingly, we can apply Theorem~\ref{thm:elashvili classification} to find that $V$ is generically $H$-stable (hence also generically $G$-stable\footnote{One way to see this is by using Corollary~\ref{cor:crit gen stab}.}), except if $(d_1,\dots,d_k;m)$ is one of the following cases:
\begin{enumerate}[label=(\alph*)]
  \item $(1,\dots,1,2,d,d;1)$ for some $d\geq2$,
  \item $(1,\dots,1,2,2;2)$,
  \item $(1,\dots,1,d,d;2)$ for some $d>2$,
  \item $(1,\dots,1,2,d;d)$ for some $d>2$.
\end{enumerate}
In case~(a), we have $m=1$ and hence $G\cong H$, so $V$ is not generically $G$-stable either.
To deal with the case that $m=2$, we observe that an s.g.p.\ for $G$ can be obtained by intersecting a generic $H$-conjugate of an s.g.p.\ for $H$ with the subgroup~$G$.
From the description of the s.g.p.'s above, we can observe the following.
In case~(b), the s.g.p.~for $G$ has dimension one (the dimension drops by one compared to $H$), while in case~(c) it has dimension $d-1$ (same as for the $H$-action).
Thus we see that $V$ is in either case not generically $G$-stable.
In contrast, in case~(d) the s.g.p.\ for $G$ is finite, so $V$ is generically $G$-stable.%
\footnote{Alternately, cases~(b), (c), and (d) follow from the results on matrix normal models in~\cite{DM-mle}.}
This concludes the proof.
\end{proof}

\begin{proof}[Proof of Theorem~\ref{thm:recursive}, part (4) for $\F = \R$]
This follows from \cite[Proposition~2.23]{DM-mle}.
\end{proof}

Finally, we need to prove that if $\rho$ is not generically semistable then it is unstable.
For~$\F=\C$, this statement is contained in Lemma~\ref{lem.loc.const}.
For $\F=\R$, it then follows from \cite[Corollary 2.22 and Proposition~2.23]{DM-mle}.
This concludes the proof of Theorem~\ref{thm:recursive}.

\section{A uniform characterization}\label{sec:uniform}
In this section we prove Theorem~\ref{thm:main-inv}, which gives a non-recursive characterization.
Following~\cite{BRVR}, we define the following quantities for positive integers $k$, $d_1,\dots,d_k$, and $m$:
\begin{align*}
  R(d_1,\dots,d_k;m) &\coloneqq m \prod_{i=1}^k d_i + \sum_{n=1}^k (-1)^n G_n(d_1,\dots,d_k),
\shortintertext{where}
  G_n(d_1,\dots,d_k) &\coloneqq \sum_{1\leq i_1 < \ldots < i_n \leq k} \gcd(d_{i_1},\dots,d_{i_n})^2,
\shortintertext{as well as}
  \gmax(d_1,\dots,d_k) &\coloneqq \max_{i<j} \gcd(d_i,d_j).
\shortintertext{and}
  \Delta(d_1,\dots,d_k; m) &\coloneqq m \prod_{i=1}^k d_i - 1 - \sum_{i=1}^k (d_i^2 - 1).
\end{align*}
By convention, we define $\gmax(d) = 1$ for any $d \in \Z_{>0}$, and we always assume that $k\geq1$.

We saw earlier that generic semistability (polystability, stability) for tensor actions is symmetric in the $d_i$'s, as well as invariant under the castling transform in part~(3) of Theorem~\ref{thm:recursive}.
It is also invariant under removing dimensions $d_i$ that are equal to one.

It is not hard to verify that the quantities $R(d_1,\dots,d_k;m)$, $\gmax(d_1,\dots,d_k)$, and $\Delta(d_1,\dots,d_k;m)$ have the same invariance properties.
Hence, to prove Theorem~\ref{thm:main-inv}, it suffices to consider the case when $(d_1,\dots,d_k;m)$ is a minimal datum, and we may also assume that the $d_i$ are sorted.
Our analysis follows the same lines as the proof of~\cite[Proposition 5.3]{BRVR}.

\begin{lemma}\label{lem:R minimal}
Suppose $(d_1,\dots,d_k;m)$ is a minimal datum, and $d_1 \leq d_2 \leq \dots \leq d_k$.
Then:
\begin{enumerate}
\item $R < 0$ if and only if $d_k > md_1d_2\cdots d_{k-1}$;
\item $R = 0$ if and only if $d_k = md_1d_2\cdots d_{k-1}$;
\item $R > 0$ if and only if $d_k \leq \frac{1}{2} md_1d_2 \cdots d_{k-1}$.
\end{enumerate}
\end{lemma}
\begin{proof}
According to Lemma~\ref{lem:min-castle}, any minimal datum satisfies either $d_k > md_1d_2\cdots d_{k-1}$, $d_k = md_1d_2\cdots d_{k-1}$, or $d_k \leq \frac{1}{2} md_1d_2 \cdots d_{k-1}$.
If $d_1 = \dots = d_k = 1$ then the lemma is immediate, since $R = m - 1$.
Otherwise, we may assume that $d_1\geq2$ by removing all dimensions equal to one.
We may also assume that $m\geq2$, since when $m=1$ the lemma is already proved in~\cite[Proposition 5.3]{BRVR}.
Finally, observe that if we prove the ``if'' directions for all three statements, then the ``only if'' directions are automatic.
Hence, we proceed to prove the ``if'' directions in all three cases under the assumptions that $d_1\geq2$ and $m\geq 2$.

Let us write $B_n$ for the terms in $G_n$ that involve~$d_k$, and $A_n = G_n(d_1,\dots,d_{k-1})$ for all other terms.
Note that $A_k=0$ and $B_1 = d_k^2$.
Thus:
\begin{align}\label{eq:convenient}
  R(d_1,\dots,d_k;m)
= m \prod_{i=1}^k d_i - d_k^2 + \sum_{n=1}^{k-1} (-1)^n (A_n - B_{n+1})
\end{align}

\emph{Case (1):} Suppose $d_k > d_1 \cdots d_{k-1} m$.
Then $d_k = d_1 \cdots d_{k-1} m + \alpha$ for some $\alpha\geq1$, and using~\eqref{eq:convenient},
\begin{align*}
  R
&= -\alpha d_k + \sum_{n=1}^{k-1} (-1)^n (A_n - B_{n+1}) \\
&= -\alpha^2 - \alpha d_1 \cdots d_{k-1} m + \sum_{n=1}^{k-1} (-1)^n (A_n - B_{n+1})
\end{align*}
Clearly, $A_n \geq B_{n+1}$ for all $n$, so we can leave out the terms for odd $n$ and obtain the bound
\begin{align*}
  R
&\leq -\alpha^2 - \alpha d_1 \cdots d_{k-1} m + \sum_{n\geq2 \text{ even}} (A_n - B_{n+1}) \\
&\leq -\alpha^2 - \alpha d_1 \cdots d_{k-1} m + \sum_{n\geq2 \text{ even}} A_n \\
&< - 2 d_1 \cdots d_{k-1} + \sum_{n\geq2 \text{ even}} A_n,
\end{align*}
using that $m\geq2$ and $\alpha\geq1$.
Now we are in the same situation as in \cite[Eq.~(9)]{BRVR} and find that~$R<0$.

\smallskip

\emph{Case (2):} Suppose $d_k = d_1 \cdots d_{k-1} m$.
Here we have $B_{n+1} = A_n$ for all $n$, so using \eqref{eq:convenient},
\begin{align*}
  R = m \prod_{i=1}^k d_i - d_k^2 = 0.
\end{align*}

\smallskip

\emph{Case (3):} Suppose $d_k \leq \frac12 d_1 \cdots d_{k-1} m$.
If $k=1$ then $d_1 \leq \frac m2$ and
\begin{align*}
  R = md_1 - d_1^2 = d_1 (m - d_1) \geq \frac{m d_1}2 > 0.
\end{align*}
We now discuss the case that $k\geq2$.
Here,
\begin{align*}
  R
&= m \prod_{i=1}^k d_i - d_k^2 + \sum_{n=1}^{k-1} (-1)^n (A_n - B_{n+1}) \\
&= \frac14 d_1^2 \cdots d_{k-1}^2 m^2 - (\frac12 d_1 \cdots d_{k-1} m - d_k)^2  + \sum_{n=1}^{k-1} (-1)^n (A_n - B_{n+1}) \\
&\geq \frac14 d_1^2 \cdots d_{k-1}^2 m^2 - (\frac12 d_1 \cdots d_{k-1} m - d_{k-1})^2  + \sum_{n=1}^{k-1} (-1)^n (A_n - B_{n+1}) \\
&= d_{k-1}^2 \left( d_1 \cdots d_{k-2} m - 1 \right) + \sum_{n=1}^{k-1} (-1)^n (A_n - B_{n+1}),
\end{align*}
where the inequality follows because $d_{k-1} \leq d_k \leq \frac{1}{2} d_1\cdots d_{k-1}m$.
Leaving out the even terms, which are non-negative since $A_n \geq B_{n+1}$, we obtain
\begin{align*}
R
&\geq d_{k-1}^2 \left( d_1 \cdots d_{k-2} m - 1 \right) + \sum_{n=1}^{k-1} (-1)^n (A_n - B_{n+1}) \\
&\geq d_{k-1}^2 \left( d_1 \cdots d_{k-2} m - 1 \right) - \sum_{n\geq1 \text{ odd}} (A_n - B_{n+1}) \\
&> d_{k-1}^2 \left( d_1 \cdots d_{k-2} m - 1 \right) - \sum_{n\geq1 \text{ odd}} A_n.
\end{align*}
Each of the $\binom{k-1}n$ GCDs contributing to $A_n$ are $\leq d_{k-1}$, so
\begin{align*}
  \sum_{n\geq1 \text{ odd}} A_n
\leq d_{k-1}^2 \sum_{n\geq1 \text{ odd}} \binom{k-1}n
= d_{k-1}^2 2^{k-2}
\end{align*}
and hence
\begin{align}\label{eq:R bound}
  R
> d_{k-1}^2 \left( d_1 \cdots d_{k-2} m - 1 - 2^{k-2} \right)
\geq d_{k-1}^2 \left( 2^{k-2} m - 1 - 2^{k-2} \right)
\end{align}
For $m\geq2$ and $k\geq2$, it holds that
\begin{align}\label{eq:inner}
  2^{k-2} m - 1 - 2^{k-2}
\geq 2^{k-2} - 1
\geq 0,
\end{align}
and hence we conclude that $R>0$.
\end{proof}

\begin{remark}
Write $\mathcal{Z}(d_1,\dots,d_k) := \sum_n (-1)^{n+1} \sum_{i_1 < i_2 < \dots < i_n} {\rm gcd}(d_{i_1}, d_{i_2},\dots,d_{i_n})$.
Then, for $d_1,\dots,d_k \in \Z_{\geq 1}$, one can interpret $\mathcal{Z}(d_1,\dots,d_k)$ as the cardinality of $\bigcup_{i=1}^k \left(\Z[\frac{1}{d_i}]/\Z\right)$ in $\Q/\Z$.
In particular, $\mathcal{Z}(d_1,\dots,d_k) \geq 0$.
Further, observe that $R(d_1,\dots,d_k;m) = m \prod_{i=1}^k d_i - \mathcal{Z}(d_1^2,\dots,d_k^2)$.
\end{remark}

An alternate and short proof of the ``if" statements in cases $(1)$ and $(2)$ in the above theorem is as follows.
Observe that the quantity $R$ is invariant under the transformation $(d_1,\dots,d_k;m) \rightarrow (d_1,\dots,d_{k-1}, d_k^*;m)$ where $d_k^* = m\prod_{i=1}^{k-1}d_i - d_k$ even in the case when some of the entries are negative or zero.
Thus in case $(1)$ we get $R(d_1,\dots,d_k;m) = R(d_1,\dots,d_{k-1},d^*_k;m) < 0$ since $md_1\dots d_{k-1} d^*_k < 0$ and $\mathcal{Z}(d_1^2,\dots,d_{k-1}^2, (d^*_k)^2) \geq 0$.
In case $(2)$, using that $\mathcal{Z}(d_1^2,\dots,d_{k-1}^2, 0) = 0$, one can deduce $R(d_1,\dots,d_k;m) = R(d_1,\dots,d_{k-1},0;m) = 0$.

Now, we can prove Theorem~\ref{thm:main-inv}.

\begin{proof}[Proof of Theorem~\ref{thm:main-inv}]
Generic semistability (polystability, stability) for tensor actions is invariant under the castling transform in part~(3) of Theorem~\ref{thm:recursive} and under permuting the dimensions~$d_i$.
The same is true for the quantities $R$, $\Delta$, and $\gmax$.
So, we can assume that~$d_1 \leq d_2 \leq \dots \leq d_k$ and that $(d_1,\dots,d_k;m)$ is a minimal datum.

\smallskip

\emph{Case (1):}
Suppose $R > 0$.
Then we know from Lemma~\ref{lem:R minimal} that $d_k \leq \frac12 md_1 d_2 \cdots d_{k-1}$.
If $k = 1$, then $d_1 \leq \frac m 2$, so we must have $m\geq 2$.
Further, $\gmax = 1$, so $R > 0$ implies that $R \geq \gmax^2$.
Finally, $\rho$ is always generically stable because the action of $\SL_{d_1}$ on $(\C^{d_1})^{\oplus m}$ is generically stable as long as $m \geq d_1$ (we have $m \geq 2 d_1$).
This concludes the proof in case that $k=1$.

Now, we deal with $k\geq2$.
We may assume that $d_1 \geq 2$ by removing all dimensions equal to one (if all $d_i = 1$ then we can reduce to the case $k=1$ discussed above).
We now distinguish two cases:
\begin{itemize}
\item $m\geq 2$:
In this case we show that $R \geq \gmax^2$ and characterize equality.
If $k>2$ then \eqref{eq:inner} is not tight, and we see from \eqref{eq:R bound} that
\begin{align*}
  R > d_{k-1}^2 \geq \gmax^2.
\end{align*}
For $k=2$, we are in the matrix case.
Since $2 \leq d_1 \leq d_2 \leq \frac12 m d_1$, we find that
\begin{align*}
  R(d_1,d_2;m)
&= md_1d_2 - d_1^2 - d_2^2 + \gcd(d_1,d_2)^2
= (md_1 - d_2) d_2 - d_1^2 + \gmax^2 \\
&\geq \frac12 md_1^2 - d_1^2 + \gmax^2
= \left( \frac m2 - 1 \right) d_1^2 + \gmax^2
\geq \gmax^2,
\end{align*}
with equality if and only if $d_1 = d_2$ and $m=2$, in which case also $\gmax = d_1 = d_2 \geq 2$.

Thus we have proved that $R \geq \gmax^2$, with equality if and only if $k=2$ and $(d_1,d_2) = (d,d)$ for some $d\geq2$.
By part~(4) of Theorem~\ref{thm:recursive}, this is precisely the case where $\rho$ is generically polystable but not generically stable (when $d_k \leq \frac12m d_1 d_2 \cdots d_{k-1}$ and $m\geq2$).

\item $m=1$:
\cite[Proposition 6.1]{BRVR} shows that in this case $\Delta \geq -2$, with equality precisely in the case that $k=3$ and $(d_1,d_2,d_3)=(2,d,d)$ for some $d\geq 2$.
(If $\Delta > -2$ then in fact $\Delta \geq 2$, but we do not need this.)
By part~(4) of Theorem~\ref{thm:recursive}, this is precisely the case where $\rho$ is generically polystable but not generically stable (when $d_k \leq \frac12m d_1 d_2 \cdots d_{k-1}$ and $m=1$).
\end{itemize}

\smallskip

\emph{Case (2):}
Suppose $R = 0$.
Then we know from Lemma~\ref{lem:R minimal} that $d_k = md_1 d_2 \cdots d_{k-1}$.
By part~(2) of Theorem~\ref{thm:recursive}, $\rho$ is generically polystable, and generically stable if and only if $d_1 = \cdots = d_{k-1} = 1$.
When $k=1$, we have $\gmax=1$ (by definition) and this condition is always satisfied.
Otherwise, $d_k = md_1d_2 \cdots d_{k-1}$ means that $\gmax = \max_{i<j} \gcd(d_i,d_j) = \max_{i<k} d_i$.
Thus, we find that in either case, $\gmax=1$ if and only if $\rho$ is generically stable.

\smallskip

\emph{Case (3):}
Suppose $R < 0$.
By Lemma~\ref{lem:R minimal}, we know that $d_k > md_1 d_2 \cdots d_{k-1}$.
Hence $\rho$ is unstable by part~(1) of Theorem~\ref{thm:recursive}.
\end{proof}

\section{Maximum likelihood estimation for tensor normal models}\label{sec:mle}
In this section, we will prove Theorem~\ref{thm:main} which characterizes the boundedness of the likelihood function and the existence and uniqueness of MLEs for the tensor normal models.

The tensor normal models are the Gaussian group models corresponding to the tensor action.
Thus the results on generic stability for tensor actions translate directly to results on maximum likelihood estimation for tensor normal models via Theorem~\ref{theo:AKRS}.
This connection is perfect for $\F = \C$, whereas some more effort is required for $\F = \R$.

A technical point to note is that $G = \smash{\prod_{i=1}^k \SL_{d_i}}$ is not a subset of $\GL(V)$, $V = \F^{d_1,\dots,d_k;m}$, which is needed to apply Theorem~\ref{theo:AKRS} verbatim.
However, this is a small issue, as we may simply replace $G$ by its homomorphic image $\rho_{d_1,\dots,d_k;m}(G)$, and note that notions of semistability, polystability, and stability are the same for both groups.

\begin{proof} [Proof of Theorem~\ref{thm:main}]
We first consider the case of $\F = \C$.
Consider the action of $G = \prod_{i=1}^k \SL_{d_i}(\C)$ on $\C^{d_1,\dots,d_k}$.
The associated Gaussian group model is $\mathcal{M}_\C(d_1,\dots,d_k)$.
Thus, Corollary~\ref{cor:mle-gen.stable} implies that Theorem~\ref{thm:main-inv} translates precisely to Theorem~\ref{thm:main}.

We now discuss the relation between the real and the complex case.
For both $\F = \R$ and $\C$, Theorem~\ref{thm:main-inv} shows that generic semistability is equivalent to generic polystability.
Further, generic semistability (resp.\ polystability) over $\C$ is equivalent to generic semistability (resp.\ polystability) over $\R$, see \cite[Proposition~2.23]{DM-mle}.
Finally, for both $\F = \R$ and $\C$, generic semistability is equivalent to almost sure boundedness of log-likelihood function because the semistable locus (over $\F$) is either empty or a (non-empty) Zariski-open subset (in particular the complement of a measure zero subset), see \cite[Corollary~2.15, Proposition~2.21, Corollary~2.22]{DM-mle}.
In fact, we claim that the following are equivalent:
\begin{enumerate}
\item $\rho_{d_1,\dots,d_k;m}$ is generically semistable for $\F = \C$.
\item $\rho_{d_1,\dots,d_k;m}$ is generically semistable for $\F = \R$.
\item $\rho_{d_1,\dots,d_k;m}$ is generically polystable for $\F = \C$.
\item $\rho_{d_1,\dots,d_k;m}$ is generically polystable for $\F = \R$.

\item For the tensor normal model $\mathcal{M}_\C(d_1,\dots,d_k)$, we have almost sure boundedness of log-likelihood function for $m$ samples.
\item For the tensor normal model $\mathcal{M}_\R(d_1,\dots,d_k)$, we have almost sure boundedness of log-likelihood function for $m$ samples.
\item For the tensor normal model $\mathcal{M}_\C(d_1,\dots,d_k)$, an MLE exists almost surely for $m$ samples.
\item For the tensor normal model $\mathcal{M}_\R(d_1,\dots,d_k)$, an MLE exists almost surely for $m$ samples.
\end{enumerate}
The equivalence of (1)---(6) was discussed above. The implications $(3) \implies (7)$ and $(4) \implies (8)$ follow from Theorem~\ref{theo:AKRS} since the complement of a Zariski-open subset has Lebesgue measure zero. Further, it is also immediate that $(7) \implies (5)$ and $(8) \implies (6)$. This shows the equivalence of all eight statements.

Moreover, $\rho_{d_1,\dots,d_k;m}$ is generically stable over $\F = \C$ if and only if the same holds for $\F = \R$, see again~\cite[Proposition~2.23]{DM-mle}.
In either case, generic stability implies the almost sure existence of a unique MLE by Theorem~\ref{theo:AKRS}.
However, the converse is not necessarily true when $\F = \R$, and this is what needs to be investigated.

To summarize, the only cases we need to further study are the cases in which $\rho_{d_1,\dots,d_k;m}$ is generically polystable but not generically stable.
According to Theorem~\ref{thm:recursive}, these are the castling equivalence classes of the minimal data below:
\begin{enumerate}
\item $d_k = md_1d_2\cdots d_{k-1}$ and $d_1 \cdots d_{k-1} > 1$.
\item $(d_1,\dots,d_k,m) = (1,1,\dots,1,d,d;2)$ with $d \geq 2$.
\item $(d_1,\dots,d_k,m) = (1,1,\dots,1,2,d,d;1)$ with $d \geq 2$.
\end{enumerate}
To conclude the proof of Theorem~\ref{thm:main}, we need to show for these we do not have the almost sure existence of a unique MLE also over $\F=\R$.
By Corollary~\ref{cor:uniq->compact} and Corollary~\ref{lem:castle compact}, it suffices to prove that in any of these three minimal cases there is a Euclidean open subset consisting of points with non-compact stabilizers for each of the above minimal data.
Note that Euclidean open subsets have positive Lebesgue measure.

For case (1), observe that the proof of Lemma~\ref{lem:stab.case2} works even when the underlying field is $\R$.
So, in fact, there is a non-empty Zariski-open subset of $V$ (in particular, a set of positive measure) where the stabilizer is isomorphic to $\prod_{i=1}^{k-1} \SL_{d_i}(\R)$, which is non-compact unless $d_1 = \cdots = d_{k-1} = 1$.

We now address case (2) and distinguish two cases:
\begin{itemize}
\item $d\geq3$:
For generic $v \in \smash{\Mat_{d,d}^2 = (\R^d \otimes \R^d)^{\oplus 2}}$, we give a sequence of elements in the stabilizer with no convergent subsequence (hence proving that the stabilizer is not compact).
It was proved in \cite[Lemma~6.2]{DM-mle} that for generic $v \in \smash{\Mat_{d,d}^2}$, there exists $(g,h) \in G_v$ such that $g$ and $h$ has eigenvalues with absolute value not equal to $1$.
Since $G_v \subseteq \SL_d \times \SL_d$, this means that $\{(g^n,h^n)\}_{n \in \Z_{>0}}$ is a sequence of elements in $G_v$ with no convergent subsequence.
Hence $G_v$ is not compact. This gives in fact a Zariski open subset consisting of points with non-compact stabilizer.

\item $d=2$:
It is easy to see that the stabilizer of $v_{a,b} = \smash{\left(\begin{psmallmatrix} 1 &0 \\ 0 & 1 \end{psmallmatrix},\begin{psmallmatrix} a & 0 \\ 0 & b \end{psmallmatrix} \right) \in \Mat_{2,2}^2}$ is not compact for any $a,b\in\R$ (cf.~the discussion below Theorem~\ref{thm:elashvili classification}).
Now, let us consider $W = \{(A,B) \in \Mat_{2,2}^2 \ | \ \det(A) \neq 0, \det(tI -A^{-1}B) \text{ has distinct real roots}\}$.
Then, it is easy to see that every $w \in W$ is in the $\SL_d \times \SL_d$ orbit of $v_{a,b}$ for an appropriate choice of $a$ and $b$ (indeed, just the eigenvalues of $A^{-1}B$).
Next, observe that $W$ is a full-dimensional semi-algebraic set, indeed it is described by one Zariski-open conditions ($\det(A) \neq 0$) and one inequality (the discriminant of $\det(tI - A^{-1}B)$ is larger than zero).
Thus, $W$ is an Euclidean-open subset (hence, a set of positive Lebesgue measure), and every point in $W$ has a non-compact stabilizer.
\end{itemize}

Finally, case (3) follows from case (2) in view of Lemma~\ref{lem:res.comp} below.
\end{proof}

\begin{lemma}\label{lem:res.comp}
Let $H \subseteq G$ be a closed subgroup of an algebraic group and let $V$ be a rational representation of $G$. Let $v \in V$. If $G_v$ is compact, then so is $H_v$.
\end{lemma}
\begin{proof}
$H_v = G_v \cap H$ is a closed subset of $G_v$ and hence compact if $G_v$ is compact.
\end{proof}

We end this section with a proof of Corollary~\ref{cor:threshold}.

\begin{proof}[Proof of Corollary~\ref{cor:threshold}]
Since Theorem~\ref{thm:main} does not differentiate between $\F = \R$ and $\F = \C$, it suffices to prove this in the case of $\F = \C$.
Here, statistical notions correspond precisely to stability notions by Corollary~\ref{cor:mle-gen.stable}, so we will make our arguments in the language of stability.
First, observe that $\lceil r \rceil \leq \mlt_b (= \mlt_e)$ because $\rho_{d_1,\dots,d_k;m}$ is unstable for unless $m \geq r$ by part~$(1)$ of Theorem~\ref{thm:main-inv}.

Now, let $c = \lceil r \rceil$, so $d_k = c d_1\cdots d_{k-1} - \alpha$ for some $0 \leq \alpha < d_1d_2\cdots d_{k-1}$.
To show $\mlt_u \leq c+1$, it suffices to show that $\rho_{d_1,\dots,d_k,c+1}$ is generically stable by Lemma~\ref{lem.increase}.

We see that $\rho_{d_1,\dots,d_k;c+1}$ is castling equivalent to $\rho_{d_1,\dots,d_{k-1},d_1d_2\cdots d_{k-1} + \alpha;c+1}$. It suffices to show that one of them is generically stable. Observe that both $A = d_k$ and $B = d_1d_2\cdots d_{k-1} + \alpha$ are larger than $d_{k-1}$, so the dimensions are already in order. Since $A + B = (c+1)d_1 \cdots d_{k-1}$, we get that either $A$ or $B$ is $\leq \frac{1}{2} (c+1)d_1\cdots d_{k-1}$. Hence, we get generic stability for $\rho_{d_1,\dots,d_k;c+1}$ by parts~(3) and (4) of Theorem~\ref{thm:recursive} unless $(d_1,\dots,d_k;c+1)$ (or $(d_1,\dots,d_{k-1},B;c+1)$) is one of $(2,d,d;1)$ or $(d,d;2)$.
The former is not possible because $c+1 \geq 2$ and the latter is not possible because $k \geq 3$.
\end{proof}

\section{Dimension of the GIT quotient}\label{sec:git quotient}
In this section, let the underlying field be $\F = \C$.
Let $V$ be a rational representation of a reductive group $G$.
Then, the GIT quotient $\gitPVG$ is defined as ${\rm Proj}(\C[V]^G)$, the projective variety associated to the ring of invariants (with its natural grading).

Given what we have computed, we can also compute the dimension of the GIT quotient for the action of $G = \prod_i \SL_{d_i}$ on $V = \F^{d_1,\dots,d_k;m}$. This relies on Rosenlicht's theorem \cite[Theorem~2]{Rosenlicht} (see also the proof of \cite[Lemma~3.1]{BRVR}).

\begin{theorem}[Rosenlicht]
Let $V$ be a rational representation of a connected semisimple group $G$.
Let $H$ be the stabilizer in general position.
Then, $\dim (\gitPVG) = \dim(\mathbb{P}V) - \dim(G) + \dim(H)$, where $\dim (\gitPVG) = -1$ if and only if $\gitPVG = \emptyset$.
\end{theorem}

For the tensor action, this means that
\begin{align}\label{eq:rosenlicht concrete}
  \dim(\gitPVG) = \Delta(d_1,\dots,d_k;m) + \dim H,
\end{align}
where $\Delta = \Delta(d_1,\dots,d_k;m) = m \prod_{i=1}^k d_i - 1 - \sum_{i=1}^k (d_i^2 - 1)$ as defined above and where $H$ is the stabilizer in general position.

\begin{proof} [Proof of Theorem~\ref{thm:git}]
By Lemma~\ref{lem.cas.inv.ring}, the dimension of the GIT quotient is invariant under castling transforms, so we may assume that $(d_1,\dots,d_k;m)$ is minimal.
We handle each case separately:

\emph{Case (1):} Suppose $R < 0$.
Then $\rho$ is unstable by Theorem~\ref{thm:main-inv}.
This means that the invariant ring is given by $\C[V]^G = \C$ and that $\gitPVG$ is empty.

\emph{Case (2):} Suppose $R = 0$.
Then $m d_1 d_2 \cdots d_{k-1} = d_k$ by Lemma~\ref{lem:R minimal}.
We identify $V \cong \Mat_{d_k,d_k}$.
For the left-right action of $\SL_{d_k} \times \SL_{d_k}$, the ring of invariants is $\C[\det]$, where $\det$ denotes the determinant polynomial.
The same is true when we restrict to the second $\SL_{d_k}$, say.
Since $\{1\} \times \SL_{d_k} \subseteq G \subseteq \SL_{d_k} \times \SL_{d_k}$, the ring of invariants for $\rho$ is also $\C[\det]$.
Thus, $\gitPVG$ is a single point.

\emph{Case (3):} Suppose $R > 0$.
Whenever $\rho$ is generically stable, \eqref{eq:rosenlicht concrete} implies that the dimension of the GIT quotient is $\Delta$ (recall that the kernel of $\rho$ is zero-dimensional), while if $\rho$ is only generically polystable we need to add the dimension of the stabilizer in general position.
There are two cases to consider:
\begin{itemize}
\item $m=1$ and $\Delta = -2$:
In this case, $k=3$ and $(d_1,d_2,d_3) = (2,d,d)$ for some $d\geq 2$, as we saw in the proof of Theorem~\ref{thm:main-inv}.
If $d=2$ then the s.g.p.\ is two-dimensional, while if $d>2$ it is $(d-1)$-dimensional (see proof of Theorem~\ref{thm:recursive}, part~(4)).
Thus, since $\gmax = d$,
\begin{align*}
  \dim(\gitPVG)
= \Delta + \dim H
= - 2 + \begin{lrdcases}
2 & \text{ if } \gmax = 2 \\
\gmax - 1 & \text{ if } \gmax > 2
\end{lrdcases}
= \max(\gmax-3,0),
\end{align*}
which is also contained in \cite[Theorem 1.2]{BRVR}.

\item $m=2$ and $R = \gmax^2 > 1$:
In this case, similarly, $k=2$ and $(d_1,d_2) = (d,d)$ for some $d\geq 2$, again by the proof of Theorem~\ref{thm:main-inv}.
If $d=2$ then the s.g.p.\ is one-dimensional, while if $d>2$ then the s.g.p.~is $(d-1)$-dimensional (see proof of Theorem~\ref{thm:recursive}, part~(4)).
Thus $\dim H = \gmax - 1$ in either case and hence
\begin{equation*}
  \dim(\gitPVG)
= \Delta + \dim H
= 1 + \left( \gmax - 1 \right)
= \gmax. \qedhere
\end{equation*}
\end{itemize}
\end{proof}

\appendix
\section{Projections}\label{app:projection}

Let $\F$ be a field. By an affine $\F$-variety, we mean the zero locus in $\F^m$ of a collection of polynomials in $\F[x_1,\dots,x_m]$.

\begin{lemma}
Suppose $X$ and $Y$ are affine $\F$ varieties, then the projection map $\pi: X \times Y \rightarrow X$ is an open map in the Zariski topology.
\end{lemma}

\begin{proof}
We have $X = \mathbb{V}(f_1,\dots,f_r)$ and $Y = \mathbb{V}(g_1,\dots,g_s)$. Now, suppose $U = \mathbb{V}(p_1,\dots,p_t)^c$ is a Zariski-open subset of $X \times Y$. Then
\begin{align*}
\pi(U) &= \{a\ |\ \exists b \in Y: (a,b) \in U\} \\
& = \{a \ |\ \exists b \in Y, 1 \leq i \leq t: p_i(a,b) \neq 0\} \\
& = \mathbb{V}(\{p_{i,b}\}_{b \in Y, 1 \leq i \leq t})^c,
\end{align*}
where $p_{i,b} = p_i(-,b)$. Thus $\pi(U)$ is Zariski-open. Note that even though $p_{i,b}$ is an infinite collection of polynomials, one can extract a finite subset with the same zero locus by the Hilbert Basis Theorem.
\end{proof}

\end{document}